\documentclass[a4paper,11pt,twoside]{article}
\usepackage{wiaspreprint}
\usepackage{amsmath,amssymb,amsthm,mathptmx,hyperref,color}
\numberwithin{equation}{section}
\newtheorem{theorem}{Theorem}
\newtheorem{remark}[theorem]{Remark}
\newtheorem{definition}[theorem]{Definition}
\newtheorem{proposition}[theorem]{Proposition}

\newtheorem{corollary}[theorem]{Corollary}
\newcommand{\f}[1]{\pmb{#1}}
\DeclareMathOperator{\N}{\mathbb{N}}
\DeclareMathOperator{\R}{\mathbb{R}}
\DeclareMathOperator{\C}{\mathcal{C}}
\DeclareMathOperator{\F}{\mathcal{F}}
\DeclareMathOperator{\AC}{\mathcal{AC}}

\DeclareMathOperator{\V}{\f H^1_{0,\sigma}}
\DeclareMathOperator{\Vd}{(\f H^{1}_{0,\sigma})^*}

\DeclareMathOperator{\Ha}{\f L^2_{\sigma}}
\DeclareMathOperator{\Hb}{\f{H}^1_0}
\DeclareMathOperator{\Hc}{\f H^2}

\DeclareMathOperator{\He}{\f{H}^1}
\DeclareMathOperator{\Le}{\f{L}^2}
\newcommand{\Hrand}[1]{{\f H^{{#1}/{2}}(\partial\Omega)}}

\DeclareMathOperator{\Sr}{\mathbb{E}}

\DeclareMathOperator{\Se}{\mathcal{S}}

\DeclareMathOperator{\ra}{\rightarrow}
\DeclareMathOperator{\de}{\text{d}}
\DeclareMathOperator{\tr}{tr}

\newcommand{\dreidots}{\text{\,\multiput(0,-2)(0,2){3}{$\cdot$}}\,\,\,\,}

\newcommand{\dreidotkom}{\text{\,\multiput(0,0)(0,2){2}{$\cdot$}\put(0,0){,}}\,\,\,\,}

\newcommand{\br}[1]{\frac{\de #1}{\de t}}
\newcommand{\pat}[2]{\frac{\partial #1}{\partial #2}}
\DeclareMathOperator{\di}{\nabla \cdot}

\newcommand{\ov}[1]{\overline{#1}}

\newcommand{\rot}[1]{[ #1 ]_{\f X}}

\DeclareMathOperator{\curl}{\nabla \times}

\newcommand{\intte}[1]{\int_{0}^T{ #1} \de t}
\newcommand{\inttet}[1]{\int_{0}^{t}{ #1} \de s}
\newcommand{\inttett}[1]{\int_{0}^{t}\left [{ #1} \right ]\de s}

\renewcommand{\ll}[1]{\langle\hspace{-0.75mm}\langle{#1}\rangle\hspace{-0.75mm}\rangle}

\DeclareMathOperator{\sym}{{sym}}
\DeclareMathOperator{\Lap}{\Delta_{\f \Lambda}}

\DeclareMathOperator{\skw}{skw}
\newcommand{\sy}[1]{(\nabla \f #1)_{{\sym}}}
\newcommand{\sk}[1]{(\nabla \f #1)_{\skw}}
\renewcommand{\t}{\partial_t  }
\newcommand{\syn}[1]{(\nabla \fn {#1})_{\sym}}
\newcommand{\skn}[1]{(\nabla \fn {#1})_{\skw}}

\newcommand{\fn}[1]{\f {{#1}}_{n}}

\newcommand{\vv}{\tilde{\f v}}
\newcommand{\dd}{\tilde{\f d}}
\newcommand{\vvn}{\tilde{\f v}}
\newcommand{\ddn}{\tilde{\f d}}

\newcommand{\syvn}{(\nabla \tilde{\f v})_{{\sym}}}

\newcommand{\skvn}{(\nabla \tilde{\f v})_{\skw}}

\renewcommand{\o}{\otimes}

\newcommand{\tq}{\tilde{\f q}}

\definecolor{lightgray}{RGB}{174,205,238}

\newcommand{\HH}{\tilde{\f H}}

\newcommand\numberthis{\addtocounter{equation}{1}\tag{\theequation}}
\begin{document}
\author[R. Lasarzik]{%				% no footnote mark; please call before \footnote{}
Robert Lasarzik\nofnmark\footnote{Weierstrass Institute \\
Mohrenstr. 39 \\ 10117 Berlin \\ Germany \\
E-Mail: robert.lasarzik@wias-berlin.de}}
	% Please use firstname.lastname@wias-berlin.de
%\\\hphantom{E-Mail:} second.wiasauthor@wias-berlin.de
%, 2. Autor\footnote{Institution 2}
%\addmark{,}{n}				% seperator "," mark relating to already existing mark

\title[Approximation  and optimal control of dissipative solutions to the Ericksen--Leslie system]{Approximation  and optimal control of dissipative solutions to the Ericksen--Leslie system
%\footnote{This work was funded by CRC 901 {\em Control of self-organizing nonlinear systems: Theoretical methods and concepts of application} (Project A8)\/.
%}
}	% in English: lowercase; start with capital letter after colon or dash; short title for page headings
\nopreprint{2535}	% preprint number
%\nopreyear{2018}	% preprint year
\selectlanguage{english}		% do not change; important for date format
%\date{January 22, 2018}			% fix date, e.g. February 22, 2017
\subjclass[2010]{35A35, 35Q35, 49J20, 76A15}	% Math. Subject Classif.
%\pacs[2008]{}				% Physics Astronomy Classif., if any
\keywords{Liquid crystal,
Ericksen--Leslie equation,
weak-strong uniqueness,
dissipative Solutions,
optimal control}
%\thanks{}				% acknowledgements; period is set automatically!
%% amsart only! abstract before maketitle
%\begin{abstract}Abstract\ldots\end{abstract}
\maketitle
%% article and other classes: abstract after maketitle
\begin{abstract}
We analyze the Ericksen--Leslie system equipped with the Oseen--Frank energy in three space dimensions. 
Recently, the author introduced the concept of  dissipative solutions. 
These solutions show several advantages in comparison to the earlier introduced measure-valued solutions. 
In this article, we argue that dissipative solutions can be numerically approximated by a relative simple scheme, which fulfills the norm-restriction on the director in every step. 
We introduce a semi-discrete scheme and derive an approximated version of the relative-energy inequality for solutions of this scheme. 
Passing to the limit in the semi-discretization, we attain dissipative solutions. 
Additionally, we introduce an optimal control scheme, show the existence of an optimal control and a possible approximation strategy. We prove that the cost functional is lower semi-continuous with respect to the convergence of this approximation and argue that
an optimal control is attained in the case that there exists 
 a solution admitting additional regularity. 
\end{abstract}
%%%%%%%%%%%%%%%%%%%%%%%%%%%%%%%%%%%%%%%%%%%%%%%%%%%%%%%%%%%%%%%%%%%%%%
%                                                                    %
%                         Start the article here                     %
%                                                                    %
%%%%%%%%%%%%%%%%%%%%%%%%%%%%%%%%%%%%%%%%%%%%%%%%%%%%%%%%%%%%%%%%%%%%%%
%%%%%%%%%%%%%%%%%%%%%%%%%%%%%%%

\setcounter{tocdepth}{2}
\tableofcontents
%%%%%%%%%%%%%%%%%%%%%%%%%%%%%%%
%% Introduction
%%%%%%%%%%%%%%%%%%%%%%%%%%%%%%%
\section{Introduction}\label{sec:intro}
Solutions to nonlinear partial differential equations are often not  explicitly computable. Therefore, numerical schemes are needed to determine a good approximation of the actual solution.
Such a numerical approximation should rely on the analysis of the continuum system, in order to resemble its properties and especially converge in some sense to a solution for vanishing discretization parameters. 
In a recent series of papers,  different generalized solution concepts for the \textit{Ericksen--Leslie system} equipped with the \textit{Oseen--Frank energy} were investigated. 
The phletropa of  different solution concepts to the Ericksen--Leslie system ranges from strong~\cite{Pruess2,Pruessnematic,hong} over weak~\cite{allgemein,unsere,linliu1}, and measure-valued~\cite{masswertig} to \textit{dissipative solutions}~\cite{diss}.

In the paper at hand, we introduce a semi-discrete scheme approximating dissipative solutions~\cite{diss} of the Ericksen--Leslie system. Additionally, an \textit{optimal control problem for dissipative solutions} to the system is investigated.

 The Ericksen--Leslie equations describe the evolution of a nematic liquid crystal under flow. 
Nematic liquid crystals are anisotropic fluids. The rod- or disk-like molecules build, or are dispersed in, a fluid and are directionally ordered. 
This ordering and its direction has a substantial influence on the properties of the material such as light scattering or rheology. This gives rise to many applications, where \textit{liquid crystal displays} are only the most prominent ones. 
Other important application possibly arise in semiconductor devices, Nanotechnology~\cite{bild}, or light-driven motors~\cite{motor}.
Due to its simplicity and its good agreement with experiments (see~\cite[Sec.~11.1 page 463]{beris}), the {Ericksen--Leslie model} is one of the most common models to describe nematic liquid crystals~\cite{sabine}. 
In nematic liquid crystals, the molecules tend to be aligned in a common direction, at least in equilibrium situations. This predominant direction is described by a unit vector, the so-called director, henceforth denoted by $\f d$. 
The director can be seen as the local average over the directions of a set of molecules contained in a local volume. 
The evolution of the flow-field is modeled by a Navier--Stokes-like equation and the alignment of the molecules is modeled by an evolution equation  with nonlinear principle part resulting from the \textit{variational derivative of a nonconvex free energy}. 

This model has been in the focus of mathematics and physics research for decades~\cite{sabine}. In the physics comunity, it has been observed in analytical~\cite{blow} as well as experimental studies~\cite{defectexperiment}  that the model predicts defects~\cite{alexander} and possibly \textit{exits the nematic phase} by developing a biaxial character, \textit{i.e.}, two predominant directions of the molecules~\cite{olmsted}. 
This gives a hint, why a global existence theory for the full Ericksen--Leslie model was missing in the mathematical literature until very recently. 
On the one hand, there have been many works proving global existence of weak solutions to simplified systems~\cite{linliu1,linliu3} or the full system~\cite{allgemein}, but always equipped with the one-constant approximation of the Oseen--Frank energy and approximating the norm restriction via a double-well potential. On the other hand, there are also some local well-posedness results for the full Oseen--Frank energy with a pointwise norm restriction~\cite{localwell},~\cite{localin3d}.

Recently, the author introduced the concept of measure-valued solutions to the full Ericksen--Leslie system equipped with the Oseen--Frank energy and point-wise norm restriction~\cite{masswertig}. This is the \textit{first global solution concept} to this model and moreover a generalization of classical solutions since these measure-valued solutions enjoy the weak-strong (or rather measure-valued-strong) uniqueness~\cite{weakstrong}, \textit{i.e.}, they coincide with a local strong solution emanating from the same initial data, as long as such a strong solution exists. 
The\textit{ expectation of the measure-valued solution} fulfills the so-called dissipative formulation (see~\cite{diss} for details). To get this formulation, the solution concept is not relaxed in terms of parametrized measures, but the equality is relaxed to an inequality. The dissipative solution concept agrees with the physical modeling since the director was initially modeled as the local average over a set of molecules and is now the average of the measure-valued solution.  Indeed, while a measure-valued solution captures every possible effect influencing the liquid crystal (possibly exiting the nematic phase during the evolution), a dissipative solution only captures the \textit{quantity of interest}, the locally averaged direction of the molecules. 
Therefore, we investigate in the paper at hand the numerical approximation of a dissipative solution via a semi-discrete scheme. 

This generalized solution concept of dissipative solutions relies on a \textit{relative energy inequality}, which compares the dissipative solution in a certain way (via the relative energy) with more regular functions solving the equation possibly only approximately (see Definition~\ref{def:diss}). 

The solutions to the introduced approximate scheme solve an associated \textit{approximate relative energy inequality} (see Theorem~\ref{thm:dis}). 
This results in a new concept of convergence of a numerical scheme: Instead of showing the convergence of solutions to a numerical scheme directly, we prove that the distance between the solution to the numerical scheme and a regular test function, measured in terms of the relative energy is bounded by certain norms of the test functions, the difference of the initial values, and how well the test functions solve the Ericksen--Leslie system. 
In case that the dissipative solution fulfills additional regularity requirements, which can be expected at least locally in time, the solutions to the numerical scheme converge to the more regular solution of the original system in a stronger sense .

The proposed semi-discrete scheme would also be an appropriate choice for the discretization of a simplified version of the Ericksen--Leslie system~(compare to the second approximation scheme in~\cite{prohl}). 
%One can not hope to get a weak solution as the limit of the sequence of solutions to this scheme since it dramatically lacks  coercivity due to its anisotropic properties.
Due to its anisotropic properties, one can not expect to attain a weak solutions as the limit of the sequence of solutions to the approximate scheme since it dramatically lacks coercivity. 
 But it is still possible to prove the convergence to a dissipative solution. This seems to be an indicator that dissipative solution are a valid concept for anisotropic systems, which often lack coercivity properties on the whole space. 

In the second part of the paper, we propose an optimal control scheme. The aim is to use an electromagnetic field to stir the evolution of the liquid crystal to a desired end state. 
The optimal control problem consists of  a convex cost functional and as a constraint the solutions should fulfill the properties of a dissipative solution to the Ericksen--Leslie equations. 
In a first result, existence of an optimal control to this problem is shown to exist via standard variational arguments.
This problem is approximated by problems consisting of the same cost functional with an restricted set of admissible controls and the constraints are approximated as introduced in the first part of this article. 
The approximation strategy is somehow contrary to the often advocated strategy \glqq{}first regularize, then optimize\grqq{}. 
Following this approach would mean to introduce the regularization and penalization terms as in the proof of the existence of measure-valued solutions (see~\cite{masswertig}). But this technique bears several disadvantages. 
Due to the high-order regularization, a high-order finite element scheme has to be adopted. Additionally, the penalization due to the double-well potential requires a fine handling of the discretization, regularization, and penalization parameters (see also~\cite{prohl}). 
The question of convergence for vanishing regularization, penalization, and discretization limit and their interchange remains widely open. In contrast to this the proposed scheme does not suffer from this shortcomings, but the sense of convergence of this scheme for vanishing discretization parameter is a very weak one.

 We can prove that the cost functional is lower semi-continuous with respect to the sequence of minimizers to the  approximate problems. Additionally, in the case that the optimal solution enjoys additional regularity, the solutions to our scheme converge to an optimal control.
The additional regularity requirement on the minimizer is enough to deduce uniqueness of the dissipative solution, such that the asserted  convergence makes sense. 

In the context of $\Gamma $ or Moscow convergence, \textit{i.e.}, in the context of lower-semi-continuity and attainability, one can assert that the cost functional is lower-semi continuous with  respect to the convergence of the approximate optimal control scheme. But the optimal solution is only known to be attainable under additional assumptions on an optimal solution. Like in the context of solvability of partial differential equations, one can talk about a weak-strong optimal control scheme in comparison to weak-strong uniqueness.

It seems remarkable that even though the control enters the system nonlinearly, weak convergence of the control is enough to go to the limit in the formulation. 

\textbf{Plan of the paper:}
In Section~\ref{sec:not}, we introduce some notation including elements of tensor calculus required for a concise presentation of the results. 
In Section~\ref{sec:gov}, the full Ericksen--Leslie model is introduced and in Section~\ref{sec:energy} the general Oseen--Frank energy as well as the energy contribution due to electromagnetic effects is given. 
Section~\ref{sec:rel} collects the definition of a dissipative solution and the main result. 
Section~\ref{sec:semiconv} provides the proof of the main result by introducing the semi-discrete scheme and approximate relative energy inequality (see Section~\ref{sec:scheme}), derive \textit{a priori} estimates, and extract converging subsequences (see Section~\ref{sec:apri}), prove the approximate relative energy estimate, and show its convergence to the continuous one (see Section~\ref{sec:disrel}).  
In the last section~(see Section~\ref{sec:opt}), we introduce the continuous optimal control problem, prove the existence of an optimal control (see Section~\ref{sec:optcon}), introduce the approximate optimal control problem and show that the cost functional is lower semi-continuous with respect to the convergence of the solutions to the approximate scheme. Under additional regularity assumptions, the solutions to the approximate problems even converge to an optimal solution of the continuous one (see Section~\ref{sec:optdis}).

%%%%%%%%%%%%%%%%%%%%%%%%%
%%%%%%%%%%%%%%%%%%%%%%%%%
%%%%%%%%%%%%%%%%%%%%%%%%%%
%%%%%%%%%%%%%%%%%%%%%%%%%%%

%%%%%%%%%%%%%%%%%%%%%%%%%%%
%% Mathematics
%%%%%%%%%%%%%%%%%%%%%%%%%%%
\subsection{Notation\label{sec:not}}
Vectors of $\R^3$ are denoted by bold small Latin letters. Matrices of $\R^{3\times 3}$ are denoted by bold capital Latin letters. We also use tensors of higher order, which are denoted by bold capital Greek letters.
Moreover, numbers are denoted be small Latin or Greek letters, and capital Latin letters are reserved for potentials.
The euclidean scalar product in $\R^3$ is denoted by a dot $ \f a \cdot \f b : = \f a ^T \f b = \sum_{i=1}^3 \f a_i \f b_i$,  for $ \f a, \f b \in \R^3$ and the Frobenius product in $\R^{3\times 3}$ by a double point $ \,\f A: \f B:= \tr ( \f A^T \f B)= \sum_{i,j=1}^3 \f A_{ij} \f B_{ij}$, for $\f A , \f B \in \R^{3\times 3}$.
Additionally, the scalar product in the space of tensors of order three is denoted by three dots 
\begin{align*}
\f \Upsilon \dreidots\, \f \Gamma : =\left [ \sum_{j,k,l=1} ^3 \f \Upsilon_{jkl} \f \Gamma_{jkl}\right ], \quad    \f \Upsilon \in \R^{3\times 3 \times 3 },  \, \f \Gamma \in \R^{3\times 3 \times 3}  .
\end{align*}
The associated norms are all denoted by $| \cdot |$, where also the norms of tensors of higher order are denoted in the same way 
\begin{align*}
| \f \Lambda|^2 := \sum_{i,j,k,l=1}^3 
\f \Lambda_{ijkl}^2\,,\quad\text{for }\f \Lambda \in \R^{3^4} \quad 
\text{and }\quad| \f \Theta |^2  := \sum_{i,j,k,l,m,n=1}^3 \f \Theta ^2_{ijklmn}\,,\quad\text{for } 
\f \Theta \in \R^{3^6}\,,
\end{align*}
respectively.
%%%%%%%%%%%%%%%%%%%%%%%%%%
%%%%%%%%%%%%%%%%%%%%%%%%%%
%%%%%%%%%%%%%%%%%%%%%%%%%%%%
Similar, we define the products of tensors of different order.
The product of a tensor of fourth and third order with a matrix  is defined by
\begin{align*}
   \f \Lambda : \f A : =\left [ \sum_{k,l=1} ^3 \f \Lambda_{ijkl} \f A_{kl}\right ]_{i,j=1}^3\,, \, 
\f \Gamma : \f A := \left [ \sum_{j,k=1}^3 \f \Gamma_{ijk}\f A_{jk}\right ]_{i=1}^3\, ,  \,
%\f \Gamma \cdot \f A := \left [ \sum_{k=1}^3 \f \Gamma_{ijk}\f A_{kl}\right ]_{i,j,l=1}^3\, ,  \,   
 \f \Lambda \in \R^{3^4 },  \, \f \Gamma\in \R^{3^  3 } , \, \f A \in \R^{3\times 3}\,.
\end{align*}

The product of a tensor of sixth order and a matrix or a tensor of third order is defined via
\begin{align*}
 \f A : \f \Theta : ={}& \left [ \sum_{i,j=1} ^3 \f A_{ij} \f \Theta_{ijklmn}  \right ]_{k,l,m,n=1}^3 , \, \f \Theta \dreidots \f \Gamma : = \left [ \sum_{l,m,n=1} ^3 \f \Theta_{ijklmn} \f \Gamma_{lmn}\right ]_{i,j,k=1}^3 ,  \,  
  \f \Theta \in \R^{3^6},\f A \in \R^{3\times 3} ,\f \Gamma \in \R^{3^  3}\,.
\end{align*}
The product of a vector and a tensor of fourth order is defined differently. The definition is adjusted to the cases of this work:
 \begin{align*}
\f a \cdot \f \Theta :={}& \left [ \sum_{k=1} ^3 \f a_{k} \f \Theta_{ijklmn}  \right ]_{i,j,l,m,n=1}^3,
\, 
  \f \Theta \in \R^{3^6},\f a \in \R^{3} \,.
\end{align*}
The standard matrix and matrix-vector multiplication is written without an extra sign for bre\-vi\-ty,
$$\f A \f B =\left [ \sum _{j=1}^3 \f A_{ij}\f B_{jk} \right ]_{i,k=1}^3 \,, \quad  \f A \f a = \left [ \sum _{j=1}^3 \f A_{ij}\f a_j \right ]_{i=1}^3\, , \quad  \f A \in \R^{3\times 3},\,\f B \in \R^{3\times3} ,\, \f a \in \R^3 .$$
The outer vector product is given by
 $\f a \otimes \f b := \f a \f b^T = \left [ \f a_i  \f b_j\right ]_{i,j=1}^3$ for two vectors $\f a , \f b \in \R^3$ and by $ \f A \o \f a := \f A \f a ^T = \left [ \f A_{ij}  \f a_k\right ]_{i,j,k=1}^3 $ for a matrix $ \f A \in \R^{3\times 3} $ and a vector $ \f a \in \R^3$. 
The symmetric and skew-symmetric parts of a matrix are given by 
$\f A_{\sym}: = \frac{1}{2} (\f A + \f A^T)$ and 
$\f A _{\skw} : = \frac{1}{2}( \f A - \f A^T)$, respectively ($\f A \in \R^{3\times  3}$).
For the product of two matrices $\f A, \f B \in \R^{3\times 3 }$, we observe
 \begin{align*}
 \f A: \f B = \f A : \f B_{\sym}\,, \quad \text{if } \f A^T= \f A\quad \text{and}\quad
  \f A: \f B = \f A : \f B_{\skw}\,, \quad \text{if } \f A^T= -\f A\, .
 \end{align*}
Furthermore, it holds $\f A^T\f B : \f C = \f B : \f A \f C$ for
$\f A, \f B, \f C \in \R^{3\times 3}$ and
$ \f a\otimes \f b : \f A = \f a \cdot \f A \f b$ for
$\f a, \f b \in \R^3$, $\f A \in \R^{3\times 3 }$  and hence $ \f a \otimes \f a : \f A = \f a \cdot \f A \f a =  \f a \cdot \f A_{\sym} \f a$.
%%%%%%%%%%%%%%%%%%%%%%%%%%%%%%%%
%%%%%%%%%%%%%%%%%%%%%%%%%%%%%%%%%
%%%%%%%%%%%%%%%%%%%%%%%%%%%%%%%%%

%%%%%%%%% Nabla operator
We use  the Nabla symbol $\nabla $  for real-valued functions $f : \R^3 \to \R$ and vector-valued functions $ \f f : \R^3 \to \R^3$ 
%as well as matrix-valued functions $\f A : \R^3 \to \R^{3\times 3}$ 
denoting
\begin{align*}
\nabla f := \left [ \pat{f}{\f x_i} \right ] _{i=1}^3\, ,\quad
\nabla \f f  := \left [ \pat{\f f _i}{ \f x_j} \right ] _{i,j=1}^3 
%\, ,\quad
%\nabla \f A  := \left [ \pat{\f A _{ij}}{ \f x_k} \right ] _{i,j,k=1}^3
\, .
\end{align*}
 The divergence of a vector-valued $ \f f : \R^3 \to \R^3$  and a matrix-valued function $\f A : \R^3 \to \R^{3\times 3}$  is defined by
\begin{align*}
\di \f f := \sum_{i=1}^3 \pat{\f f _i}{\f x_i} = \tr ( \nabla \f f)\, , \quad  \di \f A := \left [\sum_{j=1}^3 \pat{\f A_{ij}}{\f x_j}\right] _{i=1}^3\, .
\end{align*}

 Additionally, we abbreviate $\nabla \nabla $ 
 by $\nabla^2 $.
For a given tensor of fourth order, we abbreviate the associated second order operator by 
$\Delta_{\f \Lambda}\f d : =\di \f \Lambda : \nabla \f d  $ acting on functions $\f d \in \C^2(\Omega ;\R^3)$.

Throughout this paper, let $\Omega \subset \R^3$ be a bounded domain with sufficiently regular boundary~$\partial \Omega$. 
We rely on the usual notation for spaces of continuous functions, Lebesgue and Sobolev spaces. Spaces of vector-valued functions are  emphasized by bold letters, for example
$
\f L^p(\Omega) := L^p(\Omega; \R^3)$,
$\f W^{k,p}(\Omega) := W^{k,p}(\Omega; \R^3)$.
The standard inner product in $L^2 ( \Omega; \R^3)$ is just denoted by
$ (\cdot \, , \cdot )$, in $L^2 ( \Omega ; \R^{3\times 3 })$
by $(\cdot ; \cdot )$, and in $L^2 ( \Omega ; \R^{3\times 3\times 3 })$ by   $(\cdot \dreidotkom \cdot )$.
The space of smooth solenoidal functions with compact support in $\Omega$ is denoted by $\mathcal{C}_{c,\sigma}^\infty(\Omega;\R^3)$. By $\f L^p_{\sigma}( \Omega) $, $\V(\Omega)$,  and $ \f W^{1,p}_{0,\sigma}( \Omega)$, we denote the closure of $\mathcal{C}_{c,\sigma}^\infty(\Omega;\R^3)$ with respect to the norm of $\f L^p(\Omega) $, $ \f H^1( \Omega) $, and $ \f W^{1,p}(\Omega)$, respectively.
The dual space of a Banach space $V$ is always denoted by $ V^*$ and equipped with the standard norm; the duality pairing is denoted by $\langle\cdot, \cdot \rangle$. The duality pairing between $\f L^p(\Omega)$ and $\f L^q(\Omega)$ (with $1/p+1/q=1$), however, is denoted by $(\cdot , \cdot )$, $( \cdot ; \cdot )$, or $( \cdot \dreidotkom \cdot )$.

The cross product of two vectors is denoted by $\times $. We introduce the notation $ \rot{\cdot}$, which is defined via
\begin{align*}
\rot{\cdot } : \R^d \ra \R^{d\times d}\, , \quad \rot{ \f h} := \begin{pmatrix}
0& - \f h_3 &\f h_2\\
\f h_3 & 0 & - \f h_1 \\
- \f h_2 & \f h_1 & 0
\end{pmatrix}\, .
\end{align*}
The $i$-th component of the vector $\f h\in \R^3$ is denoted by $\f h_i$. 
The mapping $\rot{\cdot}$ has some nice properties, for instance
\begin{align*}
\rot{\f a}\f b = \f a \times \f b \, ,\quad \rot{\f a} ^T \rot{\f b} = (\f a \cdot \f b) I - \f b \otimes \f a\, ,
\end{align*}
for all $\f a$, $\f b \in \R^3$, where $I$ denotes the identity matrix in $\R^{3\times 3}$ or 
\begin{align*}
 \quad \rot{\f a} : \nabla \f b = \rot{\f a} : \sk b = \f a \cdot \curl \f b \, , \quad \di \rot{ \f a} = - \curl \f a \, , \quad \frac{1}{2} \rot{\curl \f a} = \sk a \,,
\end{align*}
for all $ \f a, \f b \in \C^1(\ov \Omega)$.
Displaying the cross product by this matrix makes the operation associative.

For a given Banach space $ V$, Bochner--Lebesgue spaces are denoted  by $ L^p(0,T; V)$. Moreover,  $W^{1,p}(0,T; V)$ denotes the Banach space of abstract functions in $ L^p(0,T; V)$ whose weak time derivative exists and is again in $ L^p(0,T; V)$ (see also
Diestel and Uhl~\cite[Section~II.2]{diestel} or
Roub\'i\v{c}ek~\cite[Section~1.5]{roubicek} for more details).
By  $\AC([0,T];V)$, $\C([0,T]; V) $, and $ \C_w([0,T]; V)$, we denote the spaces of abstract functions mapping $[0,T]$ into $V$ that are absolutely continuous, continuous, and continuous with respect to the weak topology in $V$, respectively.
We often omit the time interval $(0,T)$ and the domain $\Omega$ and just write, \textit{e.g.}, $L^p(\f W^{k,p})$ for brevity.
Finally, by $c>0$, we denote a generic positive constant.

\section{Model\label{sec:model}}
This section introduces the considered Ericksen--Leslie system and the associated energy. 
\subsection{Governing equations\label{sec:gov}}
%%%
Let $ \Omega$ be a bounded domain
with sufficiently regular boundary~$\partial \Omega$.
We consider the Ericksen--Leslie model as introduced in~\cite{masswertig}. 
 The governing equations  read as
\begin{subequations}\label{eq:strong}
\begin{align}
\t {\f v}  + ( \f v \cdot \nabla ) \f v + \nabla p + \di \left (\nabla \f d^T \pat{F}{\nabla \f d}( \f d , \nabla\f d )\right )- \di  \f T^L&= \f g, \label{nav}\\
\f d \times \left (\t {\f d }+ ( \f v \cdot \nabla ) \f d -\sk{v}\f d + \lambda \sy{v} \f d + \f q\right ) & =0,\label{dir}\\
\di \f v & = 0,
\\
| \f d |&=1.
\end{align}%
\end{subequations}

The variable $\f v : \ov{\Omega}\times [0,T] \ra \R^3$ denotes the velocity  of the fluid, $\f d:\ov{\Omega}\times[0,T]\ra \R^3$ represents the director, \textit{i.e.}, the orientation of the rod-like molecules, and $p:\ov{\Omega}\times [0,T] \ra\R$ denotes the pressure.
The Helmholtz free energy potential~$F$, which is described rigorously in the next section, is assumed to depend only on the director and its gradient, $F= F( \f d, \nabla \f d)$.
The free energy functional~$\mathcal{F}$  is defined by
\begin{align*}
\mathcal{F}: \He \ra \R , \quad \mathcal{F}(\f d):= \int_{\Omega} F( \f d, \nabla \f d) \de \f x \,,
\end{align*}
and $\f q$ is its variational derivative (see Furihata and Matsuo~\cite[Section 2.1]{furihata}),
\begin{subequations}\label{abkuerzungen}
\begin{align}\label{qdefq}
\f q :=\frac{\delta \mathcal{F}}{\delta \f d}(\f d) =  \pat{F}{\f d}(\f d , \nabla\f d)-\di \pat{F}{\nabla \f d}(\f d, \nabla \f d)\, .
\end{align}
The elastic part of the stress tensor, \textit{i.e.}, $\nabla \f d^T {\partial F}/{\partial \nabla \f d}$ is named after Ericksen and the dissipative 
%The Ericksen stress tensor $\f T^E$ is given by
%\begin{equation}
%\f T^E = \nabla \f d^T \pat{F}{\nabla \f d}( \f d , \nabla\f d ) \, .\label{Erik}
%\end{equation}
%The 
Leslie tensor is given by
\begin{align}
\begin{split}
\f T^L_1 ={}&  \mu_1 (\f d \cdot \sy{v}\f d )\f d \otimes \f d +\mu_4 \sy{v}
 + {(\mu_5+\mu_6)} \left (  \f d \otimes\sy{v}\f d \right )_{\sym}
\\
& +{(\mu_2+\mu_3)} \left (\f d \otimes \f e  \right )_{\sym}
 +\lambda \left ( \f d \otimes \sy{v}\f d  \right )_{\skw} + \left (\f d \otimes \f e  \right )_{\skw}\, ,
\end{split}
\label{Leslie1}
\end{align}
where
\begin{align}
\f e : = \t {\f d} + ( \f v \cdot \nabla ) \f d - \sk v\f d\, .\label{e}
\end{align}

We emphasis that Parodi's law is always assumed 
\begin{equation}
 \lambda = \mu_2 + \mu _3\,.\label{parodi}
\end{equation}
It follows from Onsager's reciprocal relation and is essential to prove the energy inequality~\eqref{prop:enrgydis}. 
%The formulation~\eqref{Leslie1} is adopted by many authors including myself~\cite{unsere}. 
Formally, equation~\eqref{dir} can be taken in the cross product with $\f d$ itself, this leads to
\begin{align*}
- ( I - \f d \o \f d ) ( \t \f d + ( \f v \cdot \nabla ) \f d - ( \sy v \f d + \lambda\sy v \f d + \f q ) =0 \,.
\end{align*}
The norm restriction $  | \f d|=1$ on the director implies $ \t | \f d |^2 = ( \f v \cdot \nabla ) | \f d |^2 = \f d \cdot \sk v \f d = 0$ such that $ {\f e} \cdot \f d =0$. Hence, we may infer $ {\f e}= -( I - \f d\o\f d) ( \lambda \sy v \f d + \f q)$. Inserting this into~\eqref{Leslie1}, yields
\begin{align}
\begin{split}
\f T^L ={}&  (\mu_1 + \lambda^2) (\f d \cdot \sy{v}\f d )\f d \otimes \f d +\mu_4 \sy{v}
 + {(\mu_5+\mu_6- \lambda^2 )} \left (  \f d \otimes\sy{v}\f d \right )_{\sym}
\\
&- \lambda ( \f d \o ( I - \f d \o \f d ) \f q) _{ \sym}- \left (\f d \otimes \f q  \right )_{\skw}\, .
\end{split}\label{Leslie}
\end{align}
We choose to work with the formulation~\eqref{Leslie}, which is equivalent to~\eqref{Leslie1}, but more suitable for our purposes. Indeed, replacing the time derivative in $\f e$ allows to write the semi-discrete scheme~(see~\eqref{eq:dis} below) as an explicit ordinary differential equation. This is essential to deduce existence of solutions to such a problem. 
% is neither essential for the reformulation nor the existence of measure-valued solutions, but is essential to prove the existence of \textit{suitable} measure-valued solutions (see Remark~\ref{rem:ex}) for which the weak-strong uniqueness property holds.

To ensure the dissipative character of the system, we assume that
\begin{align}
\begin{gathered}
 \mu_4 > 0, \quad (\mu_5+\mu_6)- \lambda ^2>0,\quad  \mu_1 +  \lambda ^2>0 
\,.
\end{gathered}\label{con}
\end{align}
\end{subequations}
The function $\f g$ incorporates external forces, \textit{e.g.}, gravity. We assume that $ \f g \in L^2(0,T; (\V)^*)$. Note that we do not include the pressure into our formulation (see Definition~\ref{def:diss} below). 

Finally, we impose boundary and initial conditions as follows:
\begin{subequations}\label{anfang}
\begin{align}
\f v(\f x, 0) &= \f v_0 (\f x) \quad\text{for } \f x \in \Omega ,
&\f v (  \f x, t ) &= \f 0  &\text{for }( t,  \f x ) \in [0,T] \times \partial \Omega , \\
\f d (  \f x, 0 ) & = \f d_0 ( \f x) \quad\text{for } \f x \in \Omega ,
&\f d (  \f x ,t ) & = \f d_1 ( \f x )  &\text{for }( t,  \f x ) \in [0,T] \times \partial \Omega .
\end{align}
\end{subequations}
We always assume that $\f d_1= \f d_0$ on $\partial \Omega$, which is a compatibility condition providing regularity.
%%%%%%%%%%%%%%%%%%%%%
%%%%%%%%%%%%%%%%%%%%%
%%%%%%%%%%%%%%%%%%%%%
%%%%%%%%%%%%%%%%%%%%%%%

\subsection{The general Oseen--Frank energy and electromagnetic field effects\label{sec:energy}}
The \textit{Oseen--Frank} energy is given by~(see~Leslie~\cite{leslie}) 
\begin{align}
F^{OF}(\f d , \nabla \f d) := \frac{K_1}{2} (\di \f d )^2 +\frac{K_2}{2}( \f d \cdot \curl \f d )^2  + \frac{K_3}{2} |\f d \times \curl \f d|^2 \,,\label{oseen}
\end{align}
where $K_1,K_2,K_3>0$.
This energy can be reformulated using the norm one restriction, to
\begin{align}
\begin{split}
 F^{OF}( \f d , \nabla \f d)&:= \frac{k_1}{2} ( \di \f d) ^2 + \frac{ k_2}{2} | \curl \f d |^2 + \frac{k_3}{2} | \f d |^2 ( \di \f d )^2 +   \frac{k_4}{2} ( \f d\cdot \curl \f d )^2  +  \frac{k_5}{2} | \f d \times \curl \f d |^2 \, ,
\end{split} \label{frei}
\end{align} 
where $ k_1=k_3=K_1/2$, $k_2={\min\{K_2,K_3\}}/{2}$, $k_4 = K_2-k_2$, and $ k_5 =K_3-k_2$ are again non-negative constants, with $k_1,k_2>0$. 
We remark that $| \f d |^2| \curl \f d |^2 = ( \f d \cdot \curl \f d )^2 + | \f d \times \curl \f d |^2 $.

To abbreviate, we define the tensor of order 4, $\f \Lambda \in \R^{3^4}$ and a tensor of order 6, $ \f \Theta \in \R^{3^6}$ via
\begin{align}
\f \Lambda_{ijkl} &: ={} k_1 \f \delta_{ij} \f \delta_{kl} + k_2 ( \f \delta_{ik}\f \delta_{jl}-\f\delta_{il}\f\delta_{jk})\,,\label{Lambda}
\intertext{and}
\f \Theta_{ijklmn} &:={}k_3 \f \delta_{ij}\f \delta_{lm}\f \delta_{kn} 
%+ k_4 \f \delta_{kn} ( \f \delta_{il}\f \delta_{jm} - \f \delta_{im}\f \delta_{jl})\notag \\& + ( k_5 -k_4) \left (  \f \delta_{il}\f \delta_{mn}\f \delta_{jk} - \f \delta_{mi}\f \delta_{ln}\f \delta_{jk} - \f \delta_{lj}\f \delta_{mn}\f \delta_{ik} + \f \delta_{jm}\f \delta_{ln}\f \delta_{ik}  \right )
%
% \\
%
%&
+ k_5  \left (  \f \delta_{il}\f \delta_{mn}\f \delta_{jk} - \f \delta_{mi}\f \delta_{ln}\f \delta_{jk} - \f \delta_{lj}\f \delta_{mn}\f \delta_{ik} + \f \delta_{jm}\f \delta_{ln}\f \delta_{ik}  \right )\notag
\\
& 
+k_4 \left (  \f \delta_{kn}\f \delta_{jm}\f \delta_{il} + \f \delta_{km}\f \delta_{jl}\f \delta_{in} + \f \delta_{kl}\f \delta_{jn}\f \delta_{im} - \f \delta_{kn}\f \delta_{jl}\f \delta_{im}- \f \delta_{km}\f \delta_{jn}\f \delta_{il} - \f \delta_{kl}\f \delta_{jm}\f \delta_{in}  \right )
\,, \label{ThetaOF}
\end{align}
respectively.
The free energy can be written as 
\begin{align*}
2 F^{OF}(\f d, \nabla \f d ) = \nabla \f d : \f \Lambda : \nabla \f d +  \nabla \f d \otimes  \f d  \dreidots \f \Theta \dreidots  \nabla \f d \otimes \f d  \,. 
\end{align*}
 
The Tensor $\f \Lambda$ is strongly elliptic, \textit{i.e.}~there is an $\eta>0$ such that $ \f a \otimes  \f b : \f \Lambda : \f a \otimes \f b   \geq \eta | \f a|^2 | \f b|^2 $ for all $\f a, \f b \in \R^3$. Indeed, it holds 
\begin{align*}
 \f a \otimes  \f b : \f \Lambda : \f a \otimes \f b  = k_1 (\f a \cdot \f b)^2 + k_2 ( | \f a |^2 | \f b|^2-( \f a \cdot \f b )^2 ) \geq \min\{k_1,k_2\} | \f a |^2 | \f b|^2\,.
\end{align*}
The second order differential operator $\Lap$ introduced by a strongly elliptic tensor is coercive on $\Hb$, \textit{i.e.}, there exists a constant $k>0$ such that 
\begin{align}\label{kill}
k\| \nabla \f d \|_{\Le}^2 \leq \frac{ 1}{2} \left ( \nabla \f d ; \f \Lambda : \nabla \f d\right )  \,
\end{align}
holds for all $\f d \in \Hb$ (see~\cite[Proposition~13.1]{chipot}). 
\begin{remark}
In comparisson to the proof of existence of measure-valued solutions (see~\cite{masswertig}), it is sufficient to assume that $k_3$, $k_4$, $k_5\geq 0$. The strict inequality is not necessary for the proof of dissipative solutions. This seems to be an artificial generalization considering the reformulation of~\eqref{oseen} to~\eqref{frei}, but it relies on the fact that the director can be bounded in the $L^\infty$-norm for the proposed scheme (see~\eqref{eq:dis} and~Proposition~\ref{prop:norm}).
\end{remark}

In the sequel of this article, we want to investigate how a dissipative solution can be controlled by means of an electromagnetic field. This seems to be a good way to control the evolution of a nematic liquid crystal. At least in the case of the famous application in liquid crystal displays, the material is controlled in this way.

The model will be extended by an electromagnetic field influencing the dynamics of the liquid crystal. 
Therefore, the model is adapted by adding an electromagnetic potential to the free energy. 
The adapted free energy potential for a magnetic field $\f H$ is given by~\cite[Section~3.2]{gennes}
\begin{align}
F_{\f H} (\f d ,\nabla \f d , \f H) = F^{OF}( \f d ,\nabla \f d ) - \frac{\chi_{\|}}{2} ( \f d \cdot \f H )^2 - \frac{\chi_{\bot}}{2} | \f d \times \f H|^2 \,,\label{electroenergy}
\end{align}
where the free energy potential $F^{OF}$ is given in~\eqref{frei}. 
The associated variational derivative is given via the definition~\eqref{qdefq} by (see~\cite{lasarzik})
\begin{align*}
%\begin{split}
\f q   ={}& - k_1 \nabla \di \f d + k_2 \curl \curl \f d - k_3\nabla (\di \f d | \f d|^2) - k_4 \di \left ( \rot{\f d} ( \f d \cdot \curl \f d ) \right ) - 4 k_5 \di \left (  ( \nabla \f d)_{\skw} \f d \o \f d\right )_{\skw}\\& + k_3 (\di \f d)^2 \f d  + k_4 ( \f d \cdot \curl \f d) \curl \f  d + 4 k_5  ( \nabla \f d)_{\skw}^T ( \nabla \f d)_{\skw}\f d- \chi_{\|} ( \f d \cdot \f H) \f H + \chi_{\bot}\f H \times (\f H \times \f d)
\\={}&-  \Lap\f d - \di \left ( \f d \cdot \f \Theta \dreidots \nabla \f d \o \f d \right ) + \nabla \f d : \f \Theta \dreidots \nabla \f d \o \f d- \chi_{\|} ( \f d \cdot \f H) \f H + \chi_{\bot}\f H \times (\f H \times \f d) \,.
%\end{split}
\numberthis\label{qdef}
\end{align*}
 Here $ \chi_{\|} $ and $ \chi_{\bot}$ denote the constants measuring the magnetic susceptibility parallel and orthogonal to the director,  both constants are negative $\chi_{\|} , \chi_{\bot} < 0$ due to the diamagnetic properties of liquid crystal (compare~\cite[Section~3.2]{gennes}). 
In usual nematic liquid crystals, it holds that $\chi_{\|}> \chi_{\bot}$, such that $|\chi_{\|}|< |\chi_{\bot}|$ which agrees with the naive perception since molecules that are not aligned should experience a bigger force.
Using the calculation rules in Section~\ref{sec:not}, we observe for $\f d $ with $|\f d|=1$ that 
\begin{align*}
- \frac{\chi_{\|}}{2} ( \f d \cdot \f H )^2 - \frac{\chi_{\bot}}{2} | \f d \times \f H|^2  = - \frac{\chi_{\|}}{2} ( \f d \cdot \f H )^2 - \frac{\chi_{\bot}}{2} ( | \f d |^2 | \f H|^2 -  (\f d \cdot \f H)^2 )  = -  \frac{\chi_{\|}-\chi_{\bot}}{2} (\f d \cdot \f H)^2 - \frac{\chi_{\bot}}{2}  | \f H|^2\,.
\end{align*}
Since $\f H$ is given and $ {\chi_{\|}-\chi_{\bot}}>0$, this energy part is minimized if $\f d$ is parallel to $\f H$. 
Clearly, the molecules are then aligned along the direction of the magnetic field $\f H$.
%We always assume that 
%\begin{align}
%\f H \in \f L^3 \,.\label{regularH}
%\end{align}
\begin{remark}\label{rem:electric1}
The influence of an electric field can be modeled similar (see~\cite[Section~3.3.1]{gennes}). For a static electric field $\f E$, the free energy potential is adapted via 
\begin{align*}
F_{\f H, \f E } (\f d ,\nabla \f d , \f H, \f E) = F( \f d ,\nabla \f d ) - \frac{\chi_{\|}}{2} ( \f d \cdot \f H )^2 - \frac{\chi_{\bot}}{2} | \f d \times \f H|^2 - \frac{\epsilon_{\|}}{8\pi} ( \f d \cdot \f E )^2 - \frac{\epsilon_{\bot}}{8\pi} | \f d \times \f E|^2 \,,
\end{align*}
where $\epsilon_{\|}$ and $\epsilon_{\bot}$ are the static dielectric constants measured along and perpendicular to the director, respectively. 
Mathematically, such an influence of an electric field can be handled similar  to the magnetic field and the calculations below (see also~\cite[Remark~5.3 and Remark~5.5]{diss}). 
Therefore, we only consider the influence of  a magnetic field in this article. 

The magnetic field is assumed to be static. If it is time-dependent, its evolution should be determined by Maxwell's equations. This would heavily impede the mathematical theory, so this additional difficulty is left for future work. Nevertheless, the magnetic field should fulfill the standard source-free assumption $\di \f H =0$. 
 \end{remark}

\section{Dissipative solvability concept and main result \label{sec:rel}}
This section is devoted to the introduction of dissipative solutions and the assertion of the main result. 
\subsection{Relative energy and dissipative solutions\label{sec:def}}
The concept of dissipative solutions heavily relies on the formulation of an appropriate relative energy for the Oseen--Frank energy. This relative energy serves as a natural comparing tool for two different solutions $(\f v , \f d)$ and $(\vv,\dd)$. 
The relative energy is defined by 
\begin{align}
\begin{split}
\mathcal{E}(\f v ,\f d , \f H | \vv , \dd, \HH  )  :
={}& \frac{1}{2} \left \| \f v - \vv \right \| _{\Le}^2 +  \frac{1}{2}\left ( \nabla \f d - \nabla \dd ; \f \Lambda : (\nabla \f  d- \nabla \dd) \right ) \\ &  +\frac{1}{2} \Big ( \nabla\f d \o \f d - \nabla \dd \o \dd  \dreidotkom \f \Theta  \dreidots (\nabla \f d \o \f d 
 - \nabla \dd  \o \dd )\Big ) \\&- \frac{\chi_{\|}}{2} \| \f d \cdot \f H - \dd \cdot \HH \|_{\Le}^2 - \frac{\chi_{\bot}}{2} \| \f d \times \f H  - \dd  \times \HH  \|_{\Le}^2\,
  \,
  \end{split}
  \label{relEnergy}
\end{align}
and the relative dissipation by
\begin{align}
\begin{split}
\mathcal{W}(\f v ,\f d | \vv , \dd ) : ={}& (\mu_1+\lambda^2) \| \f d  \cdot ( \nabla \f v  )_{\sym} \f d  - \dd  \cdot ( \nabla \vv )_{\sym } \dd  \|_{\Le} ^2+ \mu_4 \| ( \nabla \f v )_{\sym} - ( \nabla \vv )_{\sym} \|_{\Le}^2\\&+ ( \mu_5 + \mu_6 -\lambda^2 )) \| ( \nabla \f v  )_{\sym} \f d  - ( \nabla \vv )_{\sym} \dd \|_{\Le}^2  +  \| \f d  \times \f q  - \dd  \times \tq  \|_{\Le}^2 \, . 
\end{split}
\label{relW}
\end{align}
Note that the relative dissipation intrinsically depends on $\f H$ and $\HH$ too. The definition of the variational derivatives $\f q$ and $\tq$ (see~\eqref{qdef}) inherits the dependence on $\f H$ and $\HH$, respectively. The variational derivative $\tq$ is defined by~\eqref{qdef}, where $\f d$ and $\f H$ are replaced by $\dd $ and $\HH$, respectively. 
Additionally, we remark that $\chi_{\|},\chi_{\bot} <0$ such that all terms in~\eqref{relEnergy} are positive. The same holds true for the terms in~\eqref{relW}. 

Inserting the definitions of the tensors $\f \Lambda $ and $\f \Theta $,
%as well as the definition of the generalized Young measures, 
the relative energy can be expressed as
\begin{align*}
\mathcal{E}(\f v ,\f d , \f H | \vv , \dd ,\HH ) ={}&  + 
 \frac{k_1}{2}\|\di \f d   - \di  \dd \|_{L^2}^2   +  {k_2} \| (\nabla \f d )_{\skw} - ( \nabla\dd )_{\skw} \|_{\Le}^2 \notag +\frac{k_3}{2}  \| (\di \f d  ) \f d  - (\di \dd ) \dd  \|_{L^2}^2\\& + \frac{k_4}{2}\| \f d   \cdot \curl\f d   - \dd   \cdot \curl\dd   \|_{L^2}^2    \notag
+  2 k_5  \| (\nabla \f d  )_{\skw} \f d  - ( \nabla\dd )_{\skw} \dd  \|_{\Le}^2 + \frac{1}{2} \left \| \f v  - \vv  \right \| _{\Le}^2 \\
  &- \frac{\chi_{\|}}{2} \| \f d \cdot \f H - \dd \cdot \HH \|_{\Le}^2 - \frac{\chi_{\bot}}{2} \| \f d \times \f H  - \dd  \times \HH  \|_{\Le}^2\,.
\end{align*}
We always assume that $(\f v , \f d)$ and $(\vv,\dd)$ fulfill the regularity requirements~\eqref{reldiss} below such that~\eqref{relEnergy} and~\eqref{relW} are well defined. 
In the following $(\vv, \dd)$ are even assumed to fulfill~\eqref{regtest} below.

%\subsection{Dissipative solutions}
\begin{definition}[Dissipative solution]\label{def:diss}
Let $\f H \in \f L^3 $. The triple $(\f v , \f d , \f q)$ consisting of the velocity field $\f v $, the director field $\f d$ and the variational derivative $\f q$ is  said to be a dissipative  solution to~\eqref{eq:strong} with magnetic field $\f H$ if
\begin{subequations}\label{reldiss}
\begin{align}
\f v &\in \C_w(0,T;\Ha)\cap  L^2(0,T;\V)\cap W^{1,2}(0,T; ( \Hc \cap\,\V)^*)\,,
%\cap W^{1,2}(0,T; ( \f W^{1,3}_{0,\sigma}(\Omega))^*),
\\ \f d& \in \C_w(0,T;\He)\cap W^{1,2}(0,T; \f L^{3/2}) \text{ with } | \f d (\f x ,t)| =1 \text{ a.\,e.~in $\Omega\times (0,T)$}\,,\\
\f d \times \f q & \in L^2(0,T; \Le)\, 
\end{align}
\end{subequations}
and if 
\begin{align}
\begin{split}
\frac{1}{2}\mathcal{E}&  (\f v (t) , \f d(t) , \f H | \vv (t) , \dd(t) , \HH ) +\| \f H - \HH\|_{ \f L^2}^2+  {}\frac{1}{2}\int_0^t\mathcal{W} (\f v (s) , \f d(s) | \vv (s) , \dd(s))  \exp\left ({\int_s^t\mathcal{K}(\tau)\de \tau }\right )\de s\\ \leq{}&  
\mathcal{D}_0(\f v(0) , \f d(0), \f H| \vv(0), \dd(0), \HH ) \exp\left ({\int_0^t\mathcal{K}(s)\de s } \right )
%
%\Big ( \mathcal{E} (\f v (0) , \f d(0) | \vv (0) , \dd(0) +\frac{|\f \Theta|^2}{k}\|\nabla \dd \o \dd \|_{L^\infty (L^\infty)} ^2 \left \|   \f d(0)- \dd(0)    \right \|_{\Le}^2 \Big) \exp\left ({\int_0^t\mathcal{K}(s)\de s } \right )
%\\
%& +\Big ((\nabla \f d(0)-\nabla \dd (0))\o (\f d(0)-  \dd(0)) \dreidotkom \f \Theta \dreidots \nabla \dd (0)\o \dd(0) \Big) \exp\left ({\int_0^t\mathcal{K}(s)\de s }\right ) 
%\\&
\\&+ \int_0^t \left (\mathcal{A}(\vv(s), \dd(s)), \begin{pmatrix}
 \vv(s)- \f v(s)  \\ \f d(s) \times ( \tq(s) - \f q(s) + \f a(\f d(s) , \f H | \dd(s), \HH) )
  \end{pmatrix}\right ) \exp\left ({\int_s^t\mathcal{K}(\tau)\de \tau }\right )\de s  
%+ \int_0^t \left ( \mathcal{A}(s) , 
%\begin{pmatrix}
%\vv -\f v \\
%\f d
%\times (  \tq -  \f q)  + \f d   \times \dd \frac{2|\f \Theta|^2}{k}\|\nabla \dd \o \dd \|_{L^\infty(\f L^\infty)} ^2   
%\end{pmatrix}\right )
%\exp\left ({\int_s^t\mathcal{K}(\tau)\de \tau }\right )\de s  
\, 
\end{split}\label{relenin}
\end{align} 
for all  test functions $(\vv , \dd, \HH )$ with  
\begin{align}
\begin{split}
\vv &\in L^\infty(0,T; \Ha)\cap L^2(0,T; \f L^\infty)\cap  L^2 (0,T; \f W^{1,3}_{0,\sigma})\cap 
%L^1(0,T; \f W^{1,\infty})\cap 
W^{1,2}(0,T;\Vd)\, , 
 \\
\dd &\in    L^\infty(0,T;\f W^{1,3
}\cap \f L^\infty ) \cap  L^2( 0,T ; \f W^{2,3}) \cap  L^4( 0,T ; \f W^{1,6})\cap W^{1,1 }(0,T;  \f W^{1,3}\cap \f L^\infty ) \cap W^{1,2}( 0,T; \f L^3)
\end{split}\label{regtest}
\end{align}
and $| \dd| =1$   a.e.~in $\Omega\times (0,T)$, $\tr(\dd)= \f d_1$, $ \HH \in \f L^\infty$, $\tq$ given by~\eqref{qdef} with $\f d $ and $ \f H$  replaced by $\dd$ and $ \HH$, respectively,
as well as  
\begin{align}
&\intte{\left ( \f d (t) \times\left ( \t \f d(t)+( \f v(t)\cdot \nabla ) \f d(t) -  \sk{v(t)}\f d(t) + \lambda \sy{v(t)}\f d(t) + \f q(t)\right ), \f \zeta(t)\right )}=0
\label{eq:mdir}
\end{align}
for $\f \zeta \in L^2(0,T; \f L^3)$ with $| \f d( \f x ,t) |= 1$ a.e.~in $ \Omega \times (0,T)$, and an estimate for the time derivative of $\f v$, \textit{i.e.},
\begin{align}
\| \t \f v\|_{L^2(( \Hc \cap\,\V)^*)} \leq c \left (\| \f g \|_{L^2(\Vd)}, \| \f v _0 \|_{ \Ha} , \| \f d_0 \|_{ \He}, \| \f d _0 \|_{ \f L^\infty}, \| \f d _1 \|_{\Hrand{3}}  \right ) \,,\label{timev}
\end{align} 
where $c$ is a constant depending on the right-hand side $\f g$ and the initial and boundary values.
%\begin{align}
%\mathcal{E}(t) -  \frac{1}{2} \left \| \f v(t) - \vv(t) \right \| _{\Le}^2
%\leq \left ( \| \f v_0 \|_{\Le}^2 + \int_0^t\langle \f g(s) , \f v(s) \rangle\de s 
%\int_0^T\phi (t)\left ( \f q (t), \f d(t) \right ) \de t \geq  \int_0^T \phi (t)\left (2\left  ( \nabla \f d ; \f \Lambda : \nabla \f d \right ) + 4 \left ( \nabla \f d \o \f d \dreidotkom \f \Theta \dreidots \nabla \f d \o \f d \right ) \right )\de t \, 
%\end{align}
%for all $ \phi \in L^2(0,T)$. 
%
%Das ist quatsch... $\f q $ ist nur \"uber $\f d \times \f q$ definiert...

The potential $\mathcal{K}$ is given by 
\begin{align}\label{K}
\begin{split}
\mathcal{K}(s) ={}& C  \left (  \|\vv(s)\|_{\f L^\infty}^2+ \| \nabla \vv(s)\|_{\f L^{3}}^2 
% \right ) \\ & +C  \left (
+   \| \tq\|_{\f L^{3}}^2 
   %+ \| \dd (s)\|_{\f W^{1,6}}^4 
   + \| \t \dd(s)\|_{\f L^\infty} + \| \t \dd(s)\|_{\f W^{1,3}} + \| \t \dd (s) \|_{ \f L^3}^2 
%+ \| \nabla \dd \o \dd\|_{L^\infty(\f L^\infty)}^2 +1 
 \right )\,,
\end{split}
\end{align}
where $C$ is a possible large constant depending on the norms $\| \f d \|_{L^\infty(\f L^\infty)} $, $\| \dd \|_{L^\infty(\f L^{\infty})}$, which are just $1$ due to~\eqref{reldiss} and~\eqref{regtest}. 
Additionally, the constant $C$ depends on $\| \nabla \dd \|_{L^\infty(\f L^3)}$, $\| \f H \|_{\f L^3}$, and $\| \HH\|_{\f L^\infty}$. 
It is obvious that $\mathcal{K}$ is bounded in $L^1(0,T)$ due to the regularity assumptions~\eqref{regtest}.  The potential $\mathcal{K}$ can be seen as a measure for the regularity of the test functions $(\vv,\dd)$.
The potential $\mathcal{D}_0$ measures the distance in the initial point in an appropriate way, it is given by
\begin{align}
\begin{split}
\mathcal{D}_0(\f v , \f d, \f H | \vv, \dd, \HH  ) =  \mathcal{E} (\f v  , \f d ,\f H  | \vv  , \dd, \HH ) &+\frac{1}{2k} \left \|   (\f d- \dd ) \cdot  \f \Theta \dreidots \nabla \dd \o \dd   \right \|_{\Le}^2 
 \\&+\Big ((\nabla \f d-\nabla \dd )\o (\f d-  \dd) \dreidotkom \f \Theta \dreidots \nabla \dd \o \dd \Big) 
\end{split}\label{D0}
\end{align}

The operator $\mathcal{A}$ incorporates the classical formulation~\eqref{eq:strong} evaluated at the test functions $(\vv ,\dd)$ and, thus measures how well the test function approximates a strong solution to~\eqref{eq:strong}.
It is  given by
\begin{align}
\mathcal{A}(\vv, \dd ) ={}& \begin{pmatrix}
 \t {\vv}  + ( \vv \cdot \nabla ) \vv + \di (\nabla \f d ^T  \pat{F}{\nabla \f d} ( \dd  ,\nabla \dd) 
 %\left (\f\Lambda : \nabla \dd + \dd \cdot \f \Theta \dreidots \nabla \dd \o \dd \right )
 )- \di  \tilde{\f T}^L- \f g\\
\dd \times \left (\t \dd + ( \vv \cdot \nabla ) \dd -(\nabla \dd)_{\skw} \dd + \lambda (\nabla \dd)_{\sym} \dd + \tq \right ) 
\end{pmatrix} \label{A}
\intertext{and $\f a ( \f d, \f H | \dd,  \HH ) $ is given by}
\f a ( \f d ,\f H | \dd,  \HH ) :={}& \frac{1}{k } \left ( ( \dd  - \f d ) \cdot \f\Theta \dreidots \nabla \dd \o \dd \right ) : \f \Theta \dreidots \nabla \dd \o \dd+\chi_{\|} ( \HH -  \f H )(\dd \cdot \HH) - \chi_{\bot}( \HH -  \f H )  \times ( \HH \times \dd ) \,.
 \label{a}
\end{align}
%Note that the second and third summand in the definition~\eqref{a} vanish for $\f H = \HH$.
\end{definition}
\begin{remark}
\label{rem:K}
Note that the definition on $\mathcal{K}$ differs from the one in~\cite{diss}. More precisely, the norms $ \| \nabla \vv \|_{L^1( \f L^\infty)}$ and $ \| \nabla \dd \|_{L^1( \f L^\infty)}$ are missing. Consequently, we need to assume less regularity for the test functions~(compare~\eqref{regtest}) and the solution concept becomes stronger in the sense that less regularity is needed to get uniqueness. 
\end{remark}
\subsection{Main result\label{sec:main}}
\begin{theorem}\label{thm:main}
Let $\Omega \subset  \R^3$ be a bounded domain with sufficiently regular boundary~$\partial \Omega$
%of class $\C^{2,1}$ 
%a smooth domain
and let the assumptions~\eqref{parodi},~\eqref{con},~\eqref{frei}, and~\eqref{electroenergy} be fulfilled. For every $\f v_0 \in \Ha$, $ \f d_0 \in \He$ with $| \f d_0| =1 $ a.e.~in $\Omega\times (0,T)$ and $\tr (\f d_0) = \f d_1 $ with $ \f d_1 \in \f H^{s-1/2}(\partial \Omega)$ with $s\in[5/2,3]$ as well as $\f g \in L^2 (0,T;\Vd)$ and $ \f H \in \f L^3 $, there exists a
dissipative solution in the sense of Definition~\ref{def:diss}.

\end{theorem}
\begin{remark}\label{rem:}
In contrast to the proof in~\cite{diss}, the existence proof in this article is not relying on the existence of measure valued solutions and therewith a regularization and penalization technique. The result is proven via the convergence of a semi-discrete scheme and thus appropriate for a decent numerical approximation. 
The boundary condition $\f d_1$ is chosen regular enough such that $\Sr \f d_1\in \f H^s (\Omega) $ with $s\in (5/2,3]$, which  grants that $ \f H^s(\Omega) \hookrightarrow \f W^{2,3}(\Omega) $. This implies that $ \dd = \Sr \f d_1 $ is a possible test function fulfilling~\eqref{regtest} and the regularity assumptions on the test functions actually make sense. 
\end{remark}

\begin{remark}[Young measure interpretation]
The variational derivative $\f q$ can  be identified in a measure-valued sense (see~\cite{masswertig} and~\cite{diss}). There exists a generalized Gradient Young measure 
\begin{align*}
\{\nu^o _{(\f x,t)}\}&  \subset \mathcal{P} ( \R^{3\times 3})\,, \text{ a.\,e.~in $\Omega\times (0,T)$} \, ,\\
 \{m_t\} &\subset \mathcal{M}^+(\ov\Omega)\,,\text{ a.\,e.~in $ (0,T)$} \, ,  \\
 \{\nu^\infty _{(\f x,t)}\} &\subset \mathcal{P}(\ov B_3\times \Se^{3^2-1})\,,\text{ $m_t $-a.\,e.~in }\ov\Omega \text{ and a.\,e.~in }   (0,T)\, ,
\end{align*}
and a classical  Young measure $ \{ \mu _{(\f x ,t)} \} \subset \mathcal{P}(\R^3)\,, \text{ a.\,e.~in $\Omega\times (0,T)$} \, ,$
such that 
\begin{multline*}
\intte{ \ll{\nu_t, \left (\f \Upsilon :\left (\f  S    (\f \Lambda : \f S + \f h \cdot \f \Theta \dreidots \f S \o \f h )^T\right )\right ) \cdot\f \psi(t)   }}+\intte{ \ll{\nu_t, \left (\f h \times \left (\f h \cdot \f\Theta \dreidots \f S \o \f h \right )\right ) \cdot\f \psi(t)   }}
\\
+ \intte{(\rot{\f d(t)}  \left (\f \Lambda : \nabla \f d(t)  + \f d(t) \cdot \f \Theta \dreidots \nabla \f d(t)  \o \f d(t)\right ) ; \nabla \f \psi(t) )}
\\
+ \inttet{\left (\langle \mu_{t },  \f d ( t) \times \left (   \chi_{\|} \f H ( \f d( t ) \cdot \f H ) - \chi_{\bot} \f H \times ( \f H \times \f d(t) )   \right )  \rangle \right ) }
 ={}\intte{\left ( \f d (t) \times \f q(t) , \f \psi(t)\right )}  \,
\end{multline*}
for all $\f \psi \in \C_c^\infty(\Omega \times (0,T) ) $. Note that the measure $m_t$ is mutually singular in $\ov \Omega$, \textit{i.e.}, it is supported on a set of Lebesgue measure zero.   The tensor $\f \Upsilon\in \R^{3\times 3\times 3}$ is the Levi--Civita tensor defined in~\cite[Section~1.1]{diss}.
The dual pairings are defined as
\begin{align*}
\langle\mu_{(\f x ,t)},f(\f x,t)\rangle :={}&  \int_{\R^3}  f(\f x,t,\f H) \mu_{(\f x ,t)} ( \de \f H) \,\intertext{ 
%for $f \in \C( \ov \Omega \times [0,T] \times \R^3 ;\R) $   
 and}
\ll{\nu_t, f(\f h , \f S) } :={}& 
%\int_{\Omega} \langle \nu^o_{(\f x, t)},  f(\f x, t , \f d(\f x,t), \f S )  \rangle\de \f x + \int_{\ov\Omega}\langle \nu_{(\f x, t)}^\infty, \tilde{f}( \f x,t , \f h , \f S)  \rangle m_t (\de \f x)\\
%={}&
 \int_{\Omega} \int_{\R^{3\times 3} } f(\f x, t,\f d(\f x, t), \f S)  \nu^o_{(\f x, t)} ( \de \f S)\de \f x 
% \\&
 + \int_{\ov\Omega}\int_{\Se^{3^2-1} \times \ov B_3} \tilde{f} (\f x, t, \tilde{\f h} , \tilde{\f S}) \nu_{(\f x, t)}^\infty ( \de \tilde{\f S}, \de \tilde{\f h}) m_t (\de \f x)\,.
\end{align*}
%We refer to the section~\ref{sec:not} for the definition of the tensor $\f \Upsilon$.
The transformed function $\tilde{f}:\ov \Omega \times [0,T]   \times  B_3 \times B_{3 \times 3}\ra \R$, the so-called recession function is given by
\begin{align*}
\tilde{f} ( {\f x },t, \tilde{\f h} ,\tilde{\f  S} ) : = 
f ( \f x ,t, \frac{\tilde{\f h}}{\sqrt{1-|\tilde{\f h}|^2}}, \frac{\tilde{\f S}}{\sqrt{1-|\tilde{\f S}|^2}})  ( 1-|\tilde{ \f h}|^2  )( 1- | \tilde{\f S}|^2) \,.
\end{align*}
See~\cite{masswertig} for further details on generalized Gradient Young measures. 
For the existence theory, the Young measure $\mu$ is just a point measure at the considered magnetic field, \textit{i.e.}, $\mu_{(\f x ,t) } = \delta_{\f H(\f x ,t)}$. But in the case of the optimal control problem in Section~\ref{sec:opt}, we also need to relax the control in the definition of a solution.
Instead of introducing the measure-valued formulation in the definition~\ref{def:diss},  the function $\f q \in L^2 (0,T; \Le)$ itself is inserted in the definition.  
\end{remark}

\begin{remark}[Subdifferential interpretation]
The variational derivative $\f q$ can also be interpreted as an element of a suitable subdifferential of the free energy~\eqref{electroenergy}. 
The sense of this subdifferential has to be rather weak, to include the vector $\f q$. It should take into account the weak convergence result for $\{\fn q\}$ as well as the geometric properties, \textit{i.e.}, as a subdifferential of an energy on the manifold $\mathbb{S}^2$, $\f q$ should be an element of its cotangent. Following the proof of the convergence of $\{\fn q\}$, this subdifferential should be defined similar to the Bouligand subdifferential (see~\cite{scholtes}), but more general (see~\cite[Remark~3.3]{diss}). 

\end{remark}

\section{Convergence of a semi-discrete scheme\label{sec:semiconv}}
In the following we introduce a semi-discrete scheme and show its convergence to a dissipative solution.
\subsection{Semi-discrete scheme and approximate relative energy inequality\label{sec:scheme}}
We consider two general Galerkin-schemes, one for the discretization of the Navier--Stokes-like equation and one for the director equation. \\
Let $\{ W_n \} _{n\in\N} $ be such that $W_n \subset W_{n+1} $ and $W_n \subset  \Hc \cap\,\V $ for all $n \in \N$ and $\ov{\lim_{n\ra \infty}W_n} =   \Hc \cap\,\V $. Let $P_n$ be the $\Ha$-projection onto $W_n$. We additionally assume that the projection $P_n$ is stable in the $  \Hc \cap\,\V$-sense, \textit{i.e.}, there exists  a $c >0$ such that (see~\cite[Appendix Theorem~4.11, Lemma~4.26]{malek})
\begin{subequations}\label{AssumptionA}
\begin{equation}
\| P_n  \f w\| _{ \Hc\cap\,\V} \leq c \| \f w\|_{ \Hc\cap\, \V}
%\text{ as well as } \| P_n  \f w\| _{\f L^\infty} \leq c \| \f w\|_{}
\text{ for all }\f w \in  \Hc \cap\,\V \ \text{ and }n \in \N\,.\label{Asump1}
\end{equation}
Let $\{ Z_n\}_{n\in\N} $ be such that $Z_n\subset \Hb\cap \,\f L^\infty $ and $Z_n\subset Z_{n+1} $ for all $n \in \N$ and $ \ov{\lim_{n\ra \infty}Z_n} = \Hb\cap \,\f L^\infty $. Let $R_n$ be the $\Le$-projection onto $Z_n$. We additionally assume that the projection $R_n$ is $\Hb$ and $\f L^\infty$-stable, \textit{i.e.}, there exists  a $c >0$ such that 
\begin{equation}
\| R_n  \f z\| _{\Hb} \leq c \| \f z\|_{\Hb}\text{ as well as } \| R_n  \f z\| _{\f L^\infty} \leq c \| \f z\|_{\f L^\infty}\text{  for all }\f z  \in \Hb \cap\, \f L^\infty\text{ and }n \in \N\,.\label{Asump2}
\end{equation} 
\end{subequations}
%\end{minipage}
%}
%\end{center}
\begin{remark}
 As a sequence of linear spaces fulfilling the assumption on $\{W_n\}$, the spaces spanned by eigenfunctions of the Stokes operator can be chosen (see~\cite{masswertig}).  It is also possible to replace the $\Hc\cap\,\V$-regularity assumption on the spaces $W_n$ by $\V\cap\,\f L^\infty$. This is rather fulfilled by linear finite elements, but the sense of the divergence-free condition has to be redefined or added as a constraint to the system in such a case. 

For a domain $\Omega $ with sufficiently regular boundary~$\partial \Omega$, a standard Finite Element scheme, \textit{i.e.}, linear finite elements on a quasi-uniform triangulation, fulfills the above assumptions on the sequence of spaces 
%$ \{ W_n\}$ and 
$\{Z_n\}$ (see Thom\'{e}e~\cite[Lemma~5.1]{thomee}, Ciarlet~\cite[Section~3.3]{ciarlet}, or~Nitsche~\cite{nitsche}). 

\end{remark}

%%%%%%%%%%%%%%%%%%%%%%%%%%%%%%%%%%%%%%5
%%%%%%%%%%%%%%%%%%%%%%%%%%%%%%%%%%%%%%% Extension operator
%%%%%%%%%%%%%%%%%%%%%%%%%%%%%%%%%%%%%%%%
\begin{proposition}[Extension operator]
\label{prop:fort}
There exists a linear continuous operator $\Sr: \f H^{s-1/2}(\partial \Omega)  \ra \f H^{s}(\Omega)$ with $s\in [5/2,3]$, where $\Omega$ is of class $\C^{2,1}$.
This operator is the right-inverse of the trace operator, \textit{i.e.}~for all $\f g\in \f H^{{1}/{2}}(\partial \Omega)  $, it holds  $ \Sr\f g = \f g $  on $\partial\Omega$ in the sense of the trace operator.
%Additionally, it holds $ \Lap \Sr\f g = 0  \text{ in } \Omega $ and t
There exists a constant $c>0$ such that
\begin{align*}
\| \Sr\f g \|_{\f H^s(\Omega)} \leq{}& c \| \f g \|_{\Hrand{s-1}}\,  \quad \text{for }\f g \in \Hrand{s-1} \text{ for }s\in[5/2,3] \,.%\label{H4}
\end{align*}
The Sobolev exponent is chosen in such a way that the regularity of the associated test functions~\eqref{regtest} is achieved (see Remark~\ref{rem:}).
\end{proposition}
\begin{proof}
Let $\Omega$ be of class $\C^{2,1}$. The extension operator is defined via the solution operator of the problem
\begin{align*}
-\Lap \f d= 0 \quad \text{in } \Omega\,,\qquad
\f d = \f g \quad \text{on } \partial \Omega\,.
%\label{randop}
\end{align*}
This problem is uniquely solvable for a tensor enjoying the strong ellipticity (see~{McLean}~\cite[Theorem 4.10]{mclean} and~\eqref{kill}). 
The associated solution operator is linear and continuous and the regularity of this problem asserts 
 (compare~{McLean}~\cite[Theorem 4.21]{mclean})
\begin{align*}
\Sr : \f H^{s-1/2}(\partial \Omega)  \ra \f H^{s}(\Omega)   \text{ for }s\in[5/2,3] 
 \,.
\end{align*}
%We remark that $\f \Lambda $ as defined in~\eqref{Lambda} is strongly elliptic (see~\eqref{ellip}).

\end{proof}
%%%%%%%%%%%%%%%%%%%%%%%%%%%%%%%%%%%%%%%%%%%5
%%%%%%%%%%%%%%%%%%%%%%%%%%%%%%%%%%%%%%%%%%%%
%%%%%%%%%%%%%%%%%%%%%%%%%%%%%%%%%%%%%%%%%%%%
The approximate system is similar to the one in~\cite{unsere}.
Let $n\in \N$ be fixed. As usual, we consider the ansatz
\begin{align*}
\fn v ( t)  = \sum_{i=1}^n v_n^i(t)\f w_i, \quad \fn d(t) = \Sr \f d_1 + \sum_{i=1}^nd_n^i(t)\f z_i \,% \label{dar}
\end{align*}
with $( v_n^i, d_n^i) \in \AC([0,T]) $ for all $i=1,\ldots ,n$.
 
Our approximation reads as: %\\
%\begin{center}
%\fcolorbox{lightgray}{lightgray}{
%\begin{minipage}{0.9\linewidth}
%\textbf{Approximation scheme}\\
\begin{subequations}\label{eq:dis}
Find $(\fn v, \fn d ) \in\AC([0,T]; W_n \times Z_n)$ such that 
\begin{align}
\begin{split}
( \partial_t {\fn v}, \f w  ) +( ( \fn v \cdot \nabla ) \fn v, \f w  ) -(\nabla \fn d ^ T ( | \fn d |^2 I - \fn d \o \fn d )   \fn q , \f w  )+ \left (\f T^L_{n}: \nabla \f w \right )&= \left \langle  \f g   , \f w\right  \rangle,\\
\fn v(0) &= P_n \f v_0 \,,
\end{split}
\label{vdis}\\
\begin{split}
\left  ( \partial_t \fn d +  \left ((|\fn d|^2 I-\fn d \o \fn d \right  ) \left (( \fn v \cdot \nabla ) \fn d 
- \skn{v} \fn d  +  \lambda  \syn v \fn d+  \fn q \right ) , \f z\right )
 &=0 \,,
\\ \fn d(0)-  \Sr \f d_1 &= R_n(\f d_{0}- \Sr \f d_1 )\,
\end{split}\label{ddis}
\end{align}
holds for all $ \f w \in W_n$ and $ \f z \in Z_n$,
where  $ \f q_{n} $ is given by the projection of the variational derivative of the free energy
\begin{align}
\begin{split}
{\f q_{n}} : = {}&
%R_n \left ( \frac{\delta \mathcal{F}}{\delta \f d }( \fn d )\right ) = R_n
%\left (F_{\f h}( \fn d ,\nabla\fn d ) - \di F_{\f S}( \fn d , \nabla \fn d)   \right ) = {}&
R_n\left (- \Lap \fn d - \di \left (\fn d \cdot \f \Theta \dreidots \nabla \fn d \o \fn d \right ) + \nabla \fn d : \f \Theta \dreidots \nabla \fn d \o \fn d\right ) \\ & -R_n \left (  \chi_{\|} ( \fn d \cdot \f H ) \f H - \chi_{\bot} \f H \times ( \f H \times \fn d)  \right ) + \gamma \fn d  \, , \end{split}
\label{qn}
\end{align}
with $\gamma \in \R$, which is chosen to vanish for simplicity, \textit{i.e.}, $\gamma =0$, and 
\begin{align}
\begin{split}
{\f T}_{n}^L :={}&  (\mu_1+\lambda^2) (\fn d \cdot  \syn v \fn d )(\fn d \otimes \fn d)+\mu_4 \syn v  + (\mu_5+\mu_6 -\lambda^2)  \left ( \fn d \otimes \syn v \fn d \right)_{\sym}\\&- \lambda \left (\fn d \otimes (| \fn d|^2 I - \fn d \o \fn d) \fn q     \right )_{\sym} 
- | \fn d |^2 \left (\fn d \otimes \fn q   \right)_{\skw} \, 
\end{split}
\label{lesliedis}
\end{align}%
is the approximate Leslie stress.
\end{subequations}%
%\end{minipage}}
%\end{center}

Note that in comparison to formulation~\eqref{Leslie}, we replaced
% $\f e_{n}$ by $-(| \fn d |^2 I - \fn d \o \fn d)(\lambda\syn v \fn d + R_n {\f q}_{n})$. Additionally, we replaced 
 $1$ by $| \fn d|^2$ in the second line of~\eqref{lesliedis}. In the limit this should be the same, which motivates this choice for the approximate system.  
By substituting $\fn e$, we replace the time derivative $\t \fn d$ in~\eqref{vdis} and this allows to write the system~\eqref{eq:dis} as an ordinary differential equation in finite dimensions.
The solvability of this approximate system is rather standard and we refer to~\cite{unsere} for more details.
We also replaced the Ericksen-stress $\di (\nabla \f d^T \partial F / \partial \nabla \f d)  $ by $ - \nabla \fn d ( | \fn d|^2 I - \fn d \o \fn d ) \fn q$. This is motivated by the integration-by-parts formula~\cite[Equation~21]{unsere} and since for the continuous system it should hold $|\fn d |=1 $ as well as $\nabla | \fn d |^2 = 2 \nabla \fn d^T \fn d =0$.  Choosing the Ericksen--stress as in~\eqref{vdis} assures that the energy equality is valid in the approximate setting. 
The term $(|\fn d|^2I - \fn d \o \fn d )$ can also be written as $ - \fn d \times \fn d \times $ (compare to the second approximation scheme in~\cite{prohl}).

Note that there is a free parameter in the system, $\gamma$ can be chosen arbitrarily, since it does not change the other parts of the system. It can be used as a normalizing constant, for example to achieve $(\fn q , \fn d ) =0 $. 
That $\fn q$ is only defined up to an additive shift by $\fn d$ corresponds to the fact that $\fn q$ should approximate the derivative of $F$ taking values in the sphere and this derivative is an element of the cotangent space of the sphere. Since it holds $\f h \cdot \f d $ for every element $\f h$ in the tangent space at $\f d$, the cotangent space can be chosen arbitrarily in the direction $\f d$. 
\begin{theorem}\label{thm:dis}
Let $\Omega$ be a bounded domain with sufficiently regular boundary and let the assumption~\eqref{parodi},~\eqref{con} and~\eqref{frei} be fulfilled.
%
%
%for the system given by~\eqref{eq:strong}-\eqref{qdef}. \right Let additionally be $\f v_0\in \Ha$, $\f d_0\in \He$, and $\f d_1 \in \Hrand{3}$ with $ \tr \f d_0 = \f d_1 $ be given. 
For the solutions $(\fn v , \fn d )$ to the semi-discrete approximate problem~\eqref{eq:dis}, it holds under the Assumption~\eqref{AssumptionA} on the discrete spaces that 
\begin{align*}
%\begin{split}
\frac{1}{2}\mathcal{E}&(\fn v (t) , \fn d(t) , \f H | \vv(t) , \dd(t) , \HH ) + \| \f H - \HH \|_{ \f L^2}^2 + \frac{1}{2}\inttet{\mathcal{W}(\fn v (t) , \fn d(t) | \vvn(t) , \ddn (t)) \exp{\left (\int_s^t\mathcal{K}(\tau)\de\tau\right )}}\\
 \leq{}&  \mathcal{D}_0(\fn v (0) , \fn d(0) , \f H | \vv(0) , \dd(0), \HH )  
% + \frac{1}{2k} \left \|   (\fn d(0)- \ddn (0)) \cdot \f \Theta\dreidots \nabla \ddn(0) \o \ddn(0)    \right \|_{\Le}^2\right )
  \exp{\left ( \inttet{\mathcal{K}(s) }\right )} 
%\\& +\left ((\nabla \fn d(0)-\nabla \ddn  (0))\o (\fn d(0)-  \ddn (0))\dreidotkom \f \Theta \dreidots \nabla \ddn  (0)\o \ddn (0)\right )  \exp{\left ( \inttet{\mathcal{K}(s) }\right )}
\\&
 + \inttet{\left [ \left (\mathcal{A}_n(\vvn,\ddn), \begin{pmatrix}
 \vvn- \fn v  \\ \fn d \times ( \tq_n - \fn q + \f a_{n})
 %\left (
% \tq _n - \fn q+ \frac{1}{k}
%\left ( ( \ddn -\fn d ) \cdot \f \Theta \dreidots \nabla \ddn \o \ddn \right ) : \f \Theta \dreidots 
%\nabla \ddn \o \ddn 
% \right )  
% + \frac{1}{k}\| \nabla \ddn \o \ddn \|_{L^\infty( \f L^\infty)}^2  \fn d \times \ddn 
 \end{pmatrix}\right )+ \langle (I - R_n) \t \dd , \f q(\dd) - \f q(\fn d)\rangle \right ]\exp{\left (\int_s^t\mathcal{K}(\tau)\de\tau\right )}
 }
 \\
% &+ \| \f H - \HH \|_{\Le} ^2 \left ( \exp\left ( \int_0^t \mathcal{K}(s) \de s  \right )- 1 \right )\\
 & + \inttet{\left [\left ( \mathcal{A}_n(\fn v ,\fn d ) , \begin{pmatrix}
 P_n \vvn - \vvn \\ \fn d\times(R_n \f a_{n}- \f a_{n}) 
 \end{pmatrix} \right )+ \left ( \t \fn d ( | \fn d|^2 -1) , R_n \f a_n -\f a_n \right )    \right ]\exp{\left (\int_s^t\mathcal{K}(\tau)\de\tau\right )}}\\&
 + \inttet{\left (( 1 - |\ddn|^2 ) \t \ddn + \frac{ 1}{2}\t | \ddn|^2 \ddn,  \tq_n - \fn q + \f a_{n}
%\tq_n-\fn q + \frac{1}{k}\left ( \left ( \ddn -\fn d \right )\cdot \f \Theta \dreidots \nabla \ddn \o \ddn  \right ) : \f \Theta \dreidots \nabla \ddn \o \ddn 
\right )  \exp{\left (\int_s^t\mathcal{K}(\tau)\de\tau\right )}} 
%\\
% & + \inttet{ \langle (I - R_n) \t \dd , \f q(\dd) - \f q(\fn d)\rangle   \exp{\left (\int_s^t\mathcal{K}(\tau)\de\tau\right )}}
% & + \inttet{ \left ( \nabla( R_n -I)\t \dd ; \f \Lambda : \nabla \fn d + \fn d \cdot \f \Theta \dreidots \nabla \fn d \o \fn d \right )   \exp{\left (\int_s^t\mathcal{K}(\tau)\de\tau\right )}}\\
% & + \inttet{ \left ( (R_n- I) \t \dd , \nabla \fn d  : \f \Theta \dreidots \nabla \fn d - \chi_{\|} \f H ( \fn d \cdot \f H ) + \chi_{\bot} \f H\times ( \f H \times \fn d) \right )\exp{\left (\int_s^t\mathcal{K}(\tau)\de\tau\right )}}
% \end{split}
\numberthis\label{discreterelativeEnergy}
\end{align*}
for all $(\vvn, \ddn)$ fulfilling~\eqref{regtest} 
as well as $ \vv \in L^2(0,T; \Hc \cap \V)$ 
and $\tr (\dd ) =\f d_1$.
%$ \in \C^1 ([0,T]; W_m) \times\C^1 ( [0,T] ; Z_m)$ with $m \leq n$ 
Here $\mathcal{A}_n$ is given similar to~\eqref{A} by 
 \begin{align}
 {\mathcal{A}}_n( \vvn ,\ddn)  = \begin{pmatrix}
\t \vvn + ( \vvn \cdot \nabla )\vvn - \nabla \ddn^T ( | \ddn |^2 I - \ddn \o \ddn ) \tq_n - \di \tilde{\f T}^L_n- \f g \\
\ddn \times  \left  ( \t \ddn  ( \vvn \cdot \nabla ) \ddn - \skvn \ddn + \lambda \syvn \ddn + \tq _n \right )
\end{pmatrix}\,\label{An}
\end{align}
and $\tq_n$  by 
\begin{align}
\tq_n := {} R_n \left ( - \Lap \ddn - \di ( \ddn\cdot \f \Theta \dreidots \nabla \ddn \o \ddn) + \nabla \ddn : \f \Theta \dreidots \nabla\ddn \o \ddn- \chi_{\|} ( \ddn \cdot \HH ) \HH + \chi_{\bot} \HH \times ( \HH \times \ddn)  \right ) \,
\\
 \intertext{as well as $\f a_{n} $ by}
\f a_{n} :={} \frac{1}{k } \left ( ( \ddn  - \fn d ) )\cdot \f\Theta \dreidots \nabla \ddn \o \ddn \right ) : \f \Theta \dreidots \nabla \ddn \o \ddn+\chi_{\|} ( \HH -  \f H )(\ddn \cdot \HH) - \chi_{\bot}( \HH -  \f H )  \times ( \HH \times \ddn ) \,.\label{anm}
\end{align}
The term $\tilde{\f T}_n^L $ is given by $\f T^L_n $ with $\fn d$, $\fn v$, and $\fn q$ replaced by $\dd $, $\vv$, and $\tq_n$, respectively. 
Additionally, the abbreviation $ \langle\cdot  , \f q(\cdot) \rangle $ is defined for all $ \f l \in \He \cap \,\f L^\infty $ and $\f h   \in \He \cap \,\f L^\infty $ via
\begin{align*}
%\begin{split}
 \langle \f l , \f q( \f h )) :={}& \left ( \nabla \f l  ; \f \Lambda : \nabla \f h + \f h \cdot \f \Theta \dreidots \nabla \f h \o \f h \right ) +  \left ( \f l , \nabla \f h  : \f \Theta \dreidots \nabla \f h\o \f h - \chi_{\|} \f H ( \f h \cdot \f H ) + \chi_{\bot} \f H\times ( \f H \times \f h) \right )\,.
%\end{split}
\numberthis\label{tddqn}
\end{align*}

\end{theorem}
The proof of Theorem~\ref{thm:dis} is executed in Section~\ref{sec:disrel}.  Beforehand, we derive \textit{a priori} estimates In Section~\ref{sec:apri} and extract a subsequence converging in appropriate spaces to be able to go to the limit in inequality~\eqref{discreterelativeEnergy}. 
%and in section~\ref{sec:disrel}, we prove that the discrete energy inequality is fulfilled. 
% and reason that the limit is indeed a dissipative solution according to Definition~\ref{def:diss}.
\subsection{\textit{A priori} estimates and converging subsequence\label{sec:apri}}
In a first step, we show that the approximate solution obeys the norm restriction almost everywhere.
\begin{proposition}\label{prop:norm}
Let $(\fn v, \fn d)$ be a solution to the approximate system~\eqref{eq:dis}. Then it holds $| \fn d(\f x ,t)|\leq c$ a.e.~in $\Omega\times (0,T)$ and $| \f d(\f x ,t) | \ra 1$ as $n\ra \infty$ a.e.~in $\Omega\times (0,T)$.
\end{proposition}
\begin{proof}
Multiplying~\eqref{ddis} with $\fn d $ and integrating in time yields
\begin{align*}
| \fn  d(\f x , t) |^2 = | R_n (\f d_0 - \Sr \f d _0 )( \f x) + \Sr \f d_0( \f x)  |^2 \quad \text{for a.e.~} (\f x ,t) \in \Omega \times (0,T)\,.
\end{align*}
Note that $(|\fn d|^2 I -\fn d \o \fn d) \fn d = 0$.  Since $R_n$ is a stable projection in $\f L^\infty$, we observe
\begin{align*}
\| \fn d \|_{L^\infty(\f L^\infty)} \leq \| R_n (\f d_0- \Sr \f d _0 ) + \Sr \f d_0 \|_{\f L^\infty} \leq c \| \f d _0- \Sr \f d _0 \|_{\f L^\infty} + \| \Sr \f d_0 \|_{\f L^\infty} \leq c \,.
\end{align*}
The initial datum is assumed to fulfill the unit-vector restriction. 
Since $R_n(\f d_0 - \Sr \f d _0) + \Sr \f d_0 \ra \f d_0$ as $n\ra \infty$ and $| \f d_0 |=1$, it holds $ | \fn d( \f x ,t) | \ra 1$ as $ n\ra \infty$ a.e.~in $\Omega\times (0,T)$.

\end{proof}
\begin{proposition}\label{prop:enrgydis}
Let the assumptions of Theorem~\ref{thm:dis} be fulfilled and let $(\fn v , \fn d)$ be a solution to the semi-discrete problem~\eqref{eq:dis}.
Then the energy equality 
\begin{multline*}
\| \fn v(t)\|_{\Le}^2  + \mathcal{F}(\fn d(t) )+ \int_0^t \left[ (\mu_1 + \lambda^2 ) \| \fn d  \cdot \syn v \fn d \|_{L^2} ^2 + \mu_4 \| \syn v \|_{\Le}^2    \right ]\de s 
\\
+ \int_0^t\left [ (\mu_5 +\mu_6 - \lambda^2) \| \syn v\fn d\|_{\Le}^2 + \| \fn d \times  \fn q \|_{\Le}^2 \right ]\de s
\\
 = \| \fn v (0)\|_{\Le}^2 + \mathcal{F}(\fn d (0)) + \int_0^t \langle \f g, \fn v \rangle \de s 
\end{multline*}
is valid for every $t \in [0,T]$. 
We omit the dependence on $s$ under the time integral for brevity. 
\end{proposition}
\begin{proof}
The proof is very similar to the proof of Proposition~\cite[Proposition~2]{unsere}. We test equation~\eqref{vdis} with $\fn v$ and equation~\eqref{ddis} with $\fn q$ and add them up
\begin{align*}
\begin{split}
&\frac{1}{2}\br{} \| \f v_n \|_{\Le}^2  + ( \t \f d_n ,  \f q_n) - \langle \nabla \f d_n^T ( | \fn d|^2 I - \fn d \o \fn d) \f q_n , \f v_n \rangle \\&\quad  + ( ( |\fn d|^2 I - \fn d \o \fn d)  ( \f v_n \cdot \nabla )\f d_n,  \f q_n)
 + ( \f T_n^L; \nabla \f v_n)
\\
&\quad-   ((| \fn d |^2 I - \fn d \o \fn d)\skn{v} \f d_n -   \lambda \syn v \f d_n ,  \f q_n)  + \left ( \fn q , (|\fn d |^2 I -\fn d \o \fn d ) \fn q\right ) 
%\\
%& 
=  \langle \f g , \f v_n\rangle
\, .
\end{split}
\end{align*}
Note that $\fn q$ is an element of $Z_n$ since the projection $R_n$ is applied.
Inserting the Definition  of $\f T^L_n$  yields
\begin{align*}
( \f T_n^L; \nabla \f v_n) 
={}& (\mu_1 + \lambda^2 ) \| \fn d  \cdot \syn v \fn d \|_{L^2} ^2 + \mu_4 \| \syn v \|_{\Le}^2    
+ (\mu_5 +\mu_6 - \lambda^2) \| \syn v\fn d\|_{\Le}^2 \\
&+ ((|\fn d|^2I - \fn d \o \fn d)   \skn{v}\f d_n , \f q_n)  - \lambda ((| \fn d |^2 I - \fn d \o \fn d)  \syn v \f d_n ,  \f q_n)\,,
\end{align*}
where we employed  $\fn d \cdot \skn v \fn d  = 0$.
We find $ ( \t \fn d , \fn q) = \de \mathcal F(\fn d) /\de t$ by  the chain rule~\cite[Equation~33]{unsere}.
Note that the prescribed boundary values $\f d_1$ are constant in time. 
 Employing the equation $ |\fn d |^2 I -\fn d \o \fn d  = \rot{\fn d} ^T \rot {\fn d}$ and integrating in time yields the assertion.
 
\end{proof}

\begin{corollary}\label{cor:H1}
Let the assumptions of Theorem~\ref{thm:dis} be fulfilled and let $(\fn v , \fn d)$ be a solution to the semi-discrete problem~\eqref{eq:dis}. There exists a possible small $\eta >0$ and two constants $c,\,C>0$ such that
\begin{subequations}\label{coerc}
\begin{align}
\eta\|  \fn d\|_{\He}^2- c \| \f d_1 \|_{\Hrand{1}}^2 \leq{}  \frac{k_1}{2} \| \di \fn d \|_{L^2}^2 + \frac{k_2 }{2} \| \curl \fn d \|_{\Le}^2 \leq{}& \mathcal{F}(\fn d) \leq C\,,\label{coerc1}\\
  \frac{k_3 }{2} \|( \di \fn d) \fn d \|_{\Le}^2 + \frac{k_4}{2}\| \fn d \cdot \curl \fn d \|_{L^2}^2 + \frac{k_5}{2}\| \fn d \times \curl \fn d \|_{\Le}^2 \leq{}& \mathcal{F}(\fn d) \leq C \label{coerc2}
%\\
%\eta\| \fn d \|_{\f L^{12}}^4 \
%leq{}& c \| \nabla\fn d  | \fn d| \|_{\Le}^2  +c \| \f d_1 \|_{\Hrand{3}}^4\leq C  \label{coerc3}
\,. 
\end{align} 
\end{subequations}
\end{corollary}
\begin{proof}
The first inequality holds since $\f \Lambda$ (see~\eqref{Lambda}) is a strongly elliptic tensor. 
For all $\f \varphi \in \Hb$ this implies that $ \| \nabla \f \varphi(t)\|_{\Le} ^2 \leq c \left ( \nabla \f \varphi (t) ; \f \Lambda : \f \varphi (t) \right )$ holds for a.e.~$t \in (0,T)$ for some $c >0$.
Since $\fn d - \Sr \f d_1 \in \Hb$, this implies that there exists a $\eta>0$ such that
\begin{align*}
\eta \| \nabla  \fn d \|_{\Le}^2  \leq{}&  \eta \| \nabla  (\fn d- \Sr \f d_1) \|_{\Le}^2  + \eta \| \nabla \Sr \f d_1 \|_{\Le}^2 \leq  \frac{1}{4}\left ( \nabla \fn d - \nabla \Sr \f d _1   ; \,\f \Lambda : (\nabla \fn d- \nabla \Sr \f d _1 ) \right )+ \eta \| \nabla \Sr \f d_1 \|_{\Le}\\ \leq{}&  \frac{1}{2}\left ( \nabla \fn d   ; \,\f \Lambda : \nabla \fn d \right )+ c \|  \f d_1 \|_{\Hrand{1}} \,.
\end{align*}
In the last estimate we used Young's inequality and the property of the extension operator $\Sr$ (see Proposition~\ref{prop:fort}). With Poincar\'{e}'s estimate and again the extension operator $\Sr$ (see Proposition~\ref{prop:fort}) we find
\begin{align*}
\|   \fn d \|_{\Le} \leq  \|   \fn d- \Sr \f d_1 \|_{\Le}  +  \|  \Sr \f d_1 \|_{\Le}  \leq c \| \nabla(\fn d - \Sr \f d_1) \|_{\Le}^2  +  \|  \Sr \f d_1 \|_{\Le}  \leq c(\| \nabla \fn d \|_{\Le} + \| \f d_1 \|_{\Hrand{1}})\,
\end{align*}
and thus~\eqref{coerc1}. 
The estimate~\eqref{coerc2}  follows from the definition of $\mathcal{F}$ (see~\eqref{frei}) and Proposition~\ref{prop:enrgydis}. 
\end{proof}

\begin{proposition}\label{prop:time}
Let the assumptions of Theorem~\ref{thm:dis} be fulfilled and let $(\fn v , \fn d)$ be a solution to the semi-discrete problem~\eqref{eq:dis}. Then there exists a constant $c>0$ such that
\begin{align*}
\| \t \fn v \|_{L^2((\Hc \cap \, \V )^* )}+\| \t \fn d \|_{L^2 ( \f L^{3/2})}\leq c 
\end{align*}
for all $n\in\N$.
\end{proposition}
\begin{proof}
The bound on the sequence $\{ \partial_t \fn v \}$ follows from similar arguments as in~\cite[Proposition~3]{unsere}.
%Note that the regularity becomes higher due to the $\f L^\infty$ bound on the director, \textit{i.e.}, $\| \f T^L_n \|_{\Le}\leq c$ (see also~\eqref{abshtvn}).
For $\f \varphi \in L^2(0,T;\Hb)$, we test equation~\eqref{ddis} with $R_n \f \varphi$. Note that the projection $R_n$ is necessary since equation~\eqref{ddis} is only well-defined for test functions with values in $Z_n$.  It holds
\begin{align*}
%\begin{split}
\| \t \fn d \|_{L^2(\f L^{3/2})}& \leq{}\\\quad \sup_{\|\f \varphi\|_{ L^2(\f L^3 )}=1}&\Big( \| \fn d\|_{L^\infty(\f L^{\infty })}^2\left ( \| \fn v \|_{L^2 ( \f L^6)} \| \fn d \|_{L^\infty(\He)} + \| \skn v \|_{L^2(\Le)} \| \fn d \|_{L^\infty(\f L^6)} \right )  \| R_n \f \varphi\|_{L^2(\f L^3)} \\
&+ \left (| \lambda| \| \syn v \fn d \|_{L^2(\Le)} \| \fn d \|_{L^\infty(\f L^\infty )}^2 + \| \fn d \times \fn q \|_{L^2 (\Le)} \| \fn d \|_{L^\infty(\f L^\infty)} \right ) \| R_n \f \varphi \|_{L^2(\f L^2)}\Big ) \,.
%\end{split}
\numberthis\label{dtdn}
\end{align*}
%The standard embeddings in three space dimensions grant that $ \|  R_n \f \varphi\|_{L^2(\f L^3)}  \leq c \| R_n \f \varphi \|_{L^2(\Hb)} $. 
It is essential that the $\Le$-projection  $R_n$ is $\f L^\infty$-stable. This together with an interpolation argument between $\f L^2 $ and $\f L^\infty$ grants  $\| R_n \f \varphi \|_{L^2(\f L^3) }\leq \|  \f \varphi \|_{L^2(\f L^3)}$. 
%This is only valid under the Hilbert space structure of $\Hb$.
All terms on the right-hand side of the previous estimate are bounded in regard of Proposition~\ref{prop:enrgydis}, Corollary~\ref{cor:H1}, and Korn's inequality~\cite[Theorem 10.1]{mclean}. 
\end{proof}
%\subsection{Convergence to a dissipative solution\label{sec:conv}}
\begin{proposition}\label{prop:wkonv}
Let the assumptions of Theorem~\ref{thm:dis} be fulfilled and let $\{(\fn v , \fn d)\}$ be the sequence of solutions to the semi-discrete problems~\eqref{eq:dis}. Then
there exists a subsequence, which is not relabeled such that
\begin{subequations}\label{wkonv}
\begin{align}
   \fn v &\stackrel{*}{\rightharpoonup}  \f v &\quad& \text{ in } L^{\infty} (0,T;\Ha)\,,\label{wr:vstern}\\
 \fn v &\rightharpoonup  \f v &\quad& \text{ in }  L^{2} (0,T;\V)\,,\label{wr:v}\\
% \fn d &\stackrel{*}{\rightharpoonup}  \f d &\quad& \text{ in } L^{\infty} (0,T;\He)\,,\label{w:dstern}\\
\fn d \times  \fn q &\rightharpoonup \f d \times  {\f q} &\quad& \text{ in }  L^{2} (0,T;\Le)\,,\label{wr:E}\\
\syn v \fn d &\rightharpoonup   \sy v \f d &\quad& \text{ in }  L^{2} (0,T;\Le)\,,\label{wr:Dd}\\
\fn d\cdot \syn v \fn d &\rightharpoonup  \f d \cdot \sy v \f d  &\quad& \text{ in }  L^{2} (0,T;L^2)\,,\label{wr:dDd}\\
%\fd e &\rightharpoonup  {\f e} &\quad& \text{ in }  L^{2} (0,T;\Le)\,,\label{wr:e}\\
\t \fn v  &\rightharpoonup \t \f v  &\quad& \text{ in } L^2(0,T; (  \Hc \cap\,\V)^*)\, ,\label{rtimev}\\
\t\fn d &\rightharpoonup \t \f d &\quad& \text{ in } L^{2}(0,T; \f L^{3/2} ) \, ,\label{rtimed}
\\
\fn d &\stackrel{*}{\rightharpoonup}  \f d &\quad& \text{ in } L^{\infty} (0,T;\He)\,. \label{wr:ddstern}\\
%\fd v &\ra \f v &\quad& \text{ in }L^p (0,T; \Ha)\text{ for any } p \in [1,\infty)\, ,\label{sr:v}\\
\fn d &\ra \f d &\quad& \text{ in } L^q ( 0,T; \f L^r )\text{ for any } q
 \in [1,\infty)\, ,  r \in [1,6)\, ,\label{sr:d}\\
 k_3(\di\fn d)\fn d &\stackrel{*}{\rightharpoonup} k_3(\di \f d)  \f d &\quad& \text{ in } L^{\infty} (0,T;\Le)\,. \label{wr:ddiv}\\
k_4\fn d\cdot \curl \fn d  &\stackrel{*}{\rightharpoonup} k_4 \f d\cdot \curl \f d  &\quad& \text{ in } L^{\infty} (0,T;L^2)\,. \label{wr:dcurl}\\
   k_5 \fn d \times \curl \fn d &\stackrel{*}{\rightharpoonup}  k_5 \f d\times \curl \f d  &\quad& \text{ in } L^{\infty} (0,T;\Le)\,. \label{wr:ddcurl}
\end{align}
\end{subequations}
\end{proposition}
\begin{proof}
The existence of a weakly and weakly$^*$ converging subsequence follows by standard arguments from the energy equality of Proposition~\ref{prop:enrgydis} and  Proposition~\ref{prop:time}.
Note that we can bound the right-hand side of the energy equality in Proposition~\ref{prop:enrgydis}. Indeed  due to Young's inequality and  Korn's inequality~\cite[Theorem~10.1]{mclean}, we find
$$ \langle \f g , \fn v \rangle \leq   \| \fn v \|_{\Hb} \| \f g \|_{\Vd} \leq  \frac{\mu_4}{2 c^2_{\text{Korn}}} \| \fn v \|_{\Hb} ^2 + \frac{c^2_{\text{Korn}}}{2 \mu_4 }\| \f g \|_{\Vd}^2\leq  \frac{\mu_4}{2 } \| \syn v \|_{\Le} ^2 + \frac{c^2_{\text{Korn}}}{2 \mu_4 }\| \f g \|_{\Vd}^2\,,$$ where $ c_{\text{Korn}}$ is the constant due to Korn's inequality. 
The first term on the right-hand  side of the above inequality chain can be absorbed in the left-hand side of the energy equality in Proposition~\ref{prop:enrgydis}. 
Thus, every term on the left-hand side of~\eqref{prop:enrgydis} is bounded for every $t
\in [0,T]$ and for every $n\in \N$. Taking the supremum over $t$ in every term individually grants the boundedness of the terms in the above indicated norms (see~\eqref{wkonv}).

The strong convergence follows from the Lions--Aubin compactness lemma
(see Lions~\cite[Th\'eor\`eme~1.5.2]{lions}).
For $\fn d$, we observe that $\He$ is compactly embedded in $\Le$, which implies strong convergence in $L^2(0,T;\Le)$ and together with the boundedness in $L^{\infty}(0,T;\He)$ also in $L^q(0,T;\f L^r)$
for any $q\in [1,\infty)$ and any $r \in [2,6)$.  
This strong convergence allows to identify the limits in~\eqref{wr:Dd} and~\eqref{wr:dDd}.
Corollary~\ref{cor:H1} grants the weak convergences~\eqref{wr:ddiv}-\eqref{wr:ddcurl} and the strong convergence~\eqref{sr:d} allows to identify the limits.
For the limit in~\eqref{wr:E}, we initially only get that $ \fn d \times R_n \fn q \rightharpoonup \f a $, for some $\f a\in L^2(0,T; \Le) $. Due to the strong convergence of $\fn d$ to $\f d$, it holds
\begin{align*}
0 = \fn d \cdot \rot{\fn d} R_n \fn q = \fn d \cdot (\fn d \times R_n \fn q)  \rightharpoonup \f d \cdot \f a\,.
\end{align*}
The vector $\f a$ is thus point-wise orthogonal to  $\f d$ in the usual Euclidean sense. Hence, there exists a vector $\f q$ such that $\f a = \f d \times \f q$, which is the assertion of~\eqref{wr:E}. 
The constants $k_3$, $k_4$, and $k_5$ are inserted in the convergence results~\eqref{wr:ddiv}--\eqref{wr:ddcurl} since these constants can also vanish. In this case, no convergence can be deduced. 
\end{proof}
\begin{corollary}\label{cor:Cw}
Let the assumptions of Theorem~\ref{thm:dis} be fulfilled and let $\{(\fn v , \fn d)\}$ be the sequence of solutions to the semi-discrete problems~\eqref{eq:dis}. Then
there exists a subsequence, which is not relabel such that
\begin{subequations}
\label{conv:weak}
\begin{align}
   \fn v &\ra \f v &\quad& \text{ in } \C_w ([0,T];\Ha)\,,\label{cw:v}\\
\fn d &\ra  \f d &\quad& \text{ in } \C_w ([0,T];\He)\,. \label{cw:d}\\
\f \Theta \dreidots \nabla \fn d \o \fn d  &\ra \f \Theta \dreidots \nabla \f d \o \f d&\quad& \text{ in } \C_w ([0,T];\Le)\,. 
%k_3 (\di\fn d)\fn d &\ra k_3(\di \f d)  \f d &\quad& \text{ in } \C_w ([0,T];\Le)\,. \\
%k_4\fn d\cdot \curl \fn d  &\ra  k_4  \f d\cdot \curl \f d  &\quad& \text{ in } \C_w ([0,T];L^2)\,. \\
%   k_5\fn d \times \curl \fn d &\ra k_5  \f d\times \curl \f d  &\quad& \text{ in } \C_w ([0,T];\Le)\,. 
\end{align}
\end{subequations}

\end{corollary}
\begin{proof}
Due to the estimate of the time derivative in Proposition~\ref{prop:time}, we observe for the solution of the approximate Navier--Stokes-like equation that 
$$   W^{1,2}(0,T; (   \Hc \cap\,\V)^*)\hookrightarrow^c \AC([0,T];  (   \Hc \cap\,\V )^*) \hookrightarrow \C_w([0,T];  (   \Hc \cap\,\V )^*)\,.$$ The boundedness in~$L^\infty(0,T;\Ha)$ implies that there exists an $a_t\in \Ha$ and a not re-labeled subsequence such that
$\fn v (t) \rightharpoonup a_t$ in $\Ha$ as $n\ra \infty$ for a.\,e. $t\in (0,T)$.The weak convergence $ \fn v(t) \ra \f v(t) $ in $  (   \Hc \cap\,\V )^*$ as $n\ra \infty$ for a.\,e. $t\in (0,T)$ allows to identify the limit such that
$ \fn v(t) \rightharpoonup \f v(t) $ in $\Ha$ as $n\ra \infty$ for a.\,e. $t\in (0,T)$ (compare to~\cite[Page~297]{magnes}).

%with a standard lemma (see~\cite[Page~297]{magnes}) the convergence~\eqref{cw:v}. % in~$\C_w([0,T];\Ha)$.
With similar arguments, we observe the second asserted convergence of Corollary~\ref{cor:Cw}, \textit{i.e.,}~\eqref{cw:d}. 
From the convergences~\eqref{rtimed} and ~\eqref{sr:d}, we observe by the equality
\begin{multline*}
\int_\Omega | \fn d(t) - \f d(t) |^p \de \f x =\\ \int_\Omega | \fn d(0)- \f d _0 | ^p \de \f x + \int_0^t \int_\Omega| \fn d(s) - \f d(s) | ^{p-2} ( \fn d(s) - \f d(s)) \cdot( \t \fn d(s)- \f d(s)) \de \f x \de s \,
\end{multline*}
that $ \fn d \ra \f d $ in $ \C([0,T] ; L^p(\Omega))$ for $p\in [2,3)$. 
%
%
%
%The compact embedding~$\He \hookrightarrow^c \f L^5$ grants that $ \fn d \ra \f d $ in $ \C([0,T]; \f L^5)$.
 Together with~\eqref{cw:d},  this implies the convergence $ \f \Theta \dreidots \nabla \fn d \o \fn d  \ra \f \Theta \dreidots \nabla \f d \o \f d $ in $ \C_w([0,T] ; \f L^{2p/(p+2)})$ for $p\in (2,3)$.
 From the point-wise a.\,e.~bound due to Corollary~\ref{cor:H1}, we may conclude that the third asserted convergence of Corollary~\ref{cor:Cw} holds. 
%The last two convergences follow by the same line of reasoning. 
\end{proof}
\begin{remark}
The assertions of Corollary~\ref{cor:Cw} are essential to go to the limit in the relative energy. The weak-lower semi-continuity of the $\Le$-norm grants that $\liminf_{n\in\N} \mathcal{E}_n \geq \mathcal{E}$. Note that this is the only step, where we needed the boundedness of the time derivatives of the sequence  $\{\fn v\}$. 
\end{remark}

\subsection{Proof of the approximate relative energy inequality and its convergence\label{sec:disrel}}

We are going to show that the appropriate dissipative formulation~\eqref{discreterelativeEnergy}
  is fulfilled by the solution of the semi-discrete problem~\eqref{eq:dis}.% for test functions in the appropriate semi-discrete space. 
  Then we prove that in the limit a dissipative solutions  (see Definition~\ref{def:diss})
 is attained. 
Therefore, the aim is to go to the limit with the discretization parameter.% first and afterwards for the appropriate test functions $(\vvn, \ddn)$. That is why we use different indices for the discretization.

%In the \textbf{first step}, 
First, we collect associated integration-by-parts formulas. These are very similar to the ones in~\cite[Proposition~5.1]{weakstrong} and~\cite[Proposition~5.4]{diss}, but somehow simpler. There are no measures involved, since we consider a problem discretized in space and we omit the additional difficulty of introducing the cross product.
\begin{corollary}

For $(\fn v , \fn d ) \in \AC(0,T; W_n\times Z_n)$, $( \vvn, \ddn ) $ fulfilling~\eqref{regtest}, and $\f H \in L^2 , \HH \in \f L^\infty $,  it holds
\begin{subequations}\label{intdn}
\begin{align}
( \fn v ( t) , \vvn(t)) - ( \fn v (0) , \vvn(0)) ={}& \int_0^t\left[( \fn v ( s) , \t \vvn (s) ) + ( \t \fn v (s) , \vvn(s))\right ] \de s \, ,\label{intvn}\\
%\frac{1}{2}\| \fn d(t) - \ddn(t) \|_{\Le}^2- \frac{1}{2}\| \fn d(0) -\ddn(0)\|_{\Le}^2  ={}& \int_0^t \left (   \t \fn d(s) -  \t \ddn (s), \fn d (s)- \ddn (s) \right )\de s 
%% 2 \int_s^t ( \t\f d (\tau) - \t \dd(\tau), \f d ( \tau) - \dd(\tau)) \de s 
% \, , \label{intd1n}\\
\left ( \nabla \fn d(t) ;\, \f \Lambda : \nabla \ddn(t)\right ) - \left (  \nabla \fn d(0) ;\, \f \Lambda : \nabla\ddn(0)\right )
%\\ 
={} & \inttet{\left [ \left ( -  \Lap \fn d , \t \ddn  \right )
+\left ( \t \fn d , -  \Lap \ddn \right )\right ]}
%\\
%&
%- \inttet{\left ( \rot{\dd}  \Lap \dd , \rot{\f d } \t \f d  \right )}
%+ c \inttet{\mathcal{E}(}+ \inttet{\left (( \rot{\dd}-\rot{\f d } )^T ( \rot{\dd}\t \dd - \rot{\f d }\t \f d ) , - \Lap \dd \right )}
\,,
\label{intd2n}
\end{align}
\begin{align*}
%\begin{split}
\big ( \nabla \fn d(t) \o \fn d (t) \dreidotkom \f \Theta&  \dreidots \nabla \ddn(t) \o \ddn (t) \big )
-
\left ( \nabla \fn d(0) \o \fn d (0) \dreidotkom \f \Theta \dreidots \nabla \ddn(0) \o \ddn (0) \right )\\
&-
 \left ((\nabla \fn d (t)-\nabla \ddn(t)  )\o (\fn d(t) -  \ddn(t) )\right ) \dreidots \f \Theta \dreidots (\nabla \ddn(t)  \o \ddn(t) )
  \\&+\left ((\nabla \fn d (0)-\nabla \ddn(0)  )\o (\fn d(0) -  \ddn(0) )\right ) \dreidots \f \Theta \dreidots (\nabla \ddn(0)  \o \ddn(0) )
\\
\geq {}& \inttet{\left ( \t \fn d , - \di \left ( \ddn \cdot \f \Theta \dreidots \nabla \ddn \o \ddn \right )+ \nabla \ddn : \f \Theta \dreidots \nabla \ddn \o \ddn \right ) }
\\
&+\inttet{\left ( - \di \left ( \fn d \cdot \f \Theta \dreidots \nabla \fn d \o \fn d \right )+ \nabla \fn d : \f \Theta \dreidots \nabla \fn d \o \fn d, \t \ddn \right ) }
%\\
%&+ \inttet{ \ll{\nu_\tau , \f S \o \f h \dreidots \f \Theta \dreidots \left (\f \Upsilon : ( \rot \dd^T  \t \dd \o \f S )\right ) \o \f h }}
%\\&+ \inttet{ \left ( \nabla \f d \o \f d \dreidotkom \f \Theta \dreidots \rot{\f d}^T \nabla \left (\rot \dd \t \dd \right ) \o \f d \right )}
%\\
%&+ \inttet{ \ll{\nu_\tau , \f S \o \f h \dreidots \f \Theta \dreidots \f S \o \rot{\f h}^T \rot \dd \t \dd }}
%\\
%&+\inttet{\left ( \t \f d - \t \dd , - \di \left ( \dd \cdot \f \Theta \dreidots \nabla \dd \o \dd \right )+ \nabla \dd : \f \Theta \dreidots \nabla \dd \o \dd \right ) } 
\\& -c  \inttet{( (1 + \| \ddn \|_{L^\infty(\f W^{1,3})})\| \t \ddn(s)\|_{L^\infty}+  \| \nabla \t \ddn(s)\|_{\f L^3}) \mathcal{E}(\fn v ,\fn d, 0 | \vvn , \ddn,0 )}\,,
%\end{split}
\numberthis\label{intddn}
\end{align*}
\begin{align*}
%\begin{split}
&\left \| ( \fn d(t) - \ddn(t) ) \cdot \f \Theta \dreidots \nabla \ddn(t) \o \ddn(t) \right \|_{\Le}^2- \left \| ( \fn d(0) - \ddn(0) ) \cdot \f \Theta \dreidots \nabla \ddn(0) \o \ddn(0) \right \|_{\Le}^2 \\
& \leq  2 \int_0^t \left (    \partial_s \fn d - \partial_s \ddn, \left (   (  \fn d -  \ddn ) \cdot \f \Theta \dreidots \nabla \ddn \o \ddn  \right ) : \f \Theta \dreidots \nabla \ddn \o \ddn  \right ) \de s  + \int_0^t \mathcal{K} \mathcal{E}( \fn v , \fn d , 0 | \vvn , \ddn, 0) \de s \,,
%\end{split}
\numberthis\label{intd1n}
\end{align*}
 and
\begin{align*}
%\begin{split}
& \chi_{\|}  \left (\fn d(t) \cdot \f H , \ddn (t)\cdot \HH \right ) + \chi_{\bot} \left ( \fn d(t) \times \f H , \ddn (t)\times \HH \right )
 -\chi_{\|}  \left (\fn d(0) \cdot \f H , \ddn (0)\cdot \HH \right ) - \chi_{\bot} \left ( \fn d(0) \times \f H , \ddn (0)\times \HH \right )  \\
&  \geq \int_0^t \left [  \left ( \t \fn d , \chi_{\|}  \HH (\ddn \cdot \HH) - \chi_{\bot} \HH \times ( \HH \times \ddn ) \right)  + \left ( \t \ddn ,\chi_{\|}  \f H ( \fn d \cdot \f H ) - \chi_{\bot}  \f H \times ( \f H \times \fn d)   \right )   \right ]\de s \\
&  \quad - \int_0^t   \|  \t \ddn\|_{\f L^\infty} \mathcal{E}(\fn v ,\fn d , \f H | \vvn , \ddn, \HH  )   \de s \\
&\quad + \int_0^t   \left [ \left ( \t \fn d -\t \ddn  , \chi_{\|} ( \f H -  \HH )(\ddn \cdot \HH) - \chi_{\bot}( \f H -  \HH)  \times ( \HH \times \ddn ) \right)  -  \|  \t \ddn\|_{\f L^\infty} \| \HH - \f H\|_{\Le} ^2\right ]   \de s \,,
% \\
%&\quad - \int_0^t  \|  \t \dd\|_{\f L^\infty} \left  (-\chi_{\|} \| \f d \cdot \f H- \dd \cdot \HH  \|_{\Le}^2  - \chi_{\bot} \| \f H \times \f d- \HH \times \dd \|^2_{\Le}   + c \| \HH - \f H\|_{\Le} ^2 \right )   \de s \,
%\end{split}
\numberthis\label{intH}
\end{align*}
for all $t\in [0,T]$.
\end{subequations}
\end{corollary} 
\begin{proof}

The Integration by parts formulas~\eqref{intvn},~\eqref{intd2n} and~\eqref{intddn} are rather standard and simpler than the ones in~\cite[Proposition~5.1]{weakstrong} and~\cite[Proposition~5.3]{diss}. Note that the regularity in space is no problem anymore due to the discretization. For a proof in the continuous case, involving measures, we refer to~\cite{lasarzik}.  

The integration-by-parts formula~\eqref{intd1n} and~\eqref{intH} can be proven in a similar fashion. 
We start with formula~\eqref{intd1n}. 
The fundamental theorem of calculus  grants that 
\begin{align*}
&\left \| ( \fn d(t) - \ddn(t) ) \cdot \f \Theta \dreidots \nabla \ddn(t) \o \ddn(t) \right \|_{\Le}^2- \left \| ( \fn d(0) - \ddn(0) ) \cdot \f \Theta \dreidots \nabla \ddn(0) \o \ddn(0) \right \|_{\Le}^2 
\\
&= \int_0^t \partial_s \left  | ( \fn d(s) - \ddn(s) ) \cdot \f \Theta \dreidots \nabla \ddn(s) \o \ddn(s) \right |^2 \de s 
\,.
\end{align*}
The chain rule implies that 
\begin{align*}
\partial_s& \left  | ( \fn d(s) - \ddn(s) ) \cdot \f \Theta \dreidots \nabla \ddn(s) \o \ddn(s) \right |^2\\&={} 2 \left (   ( \partial_s \fn d(s) - \partial_s \ddn(s) ) \cdot \f \Theta \dreidots \nabla \ddn(s) \o \ddn(s)  \right ) :\left (   (  \fn d(s) -  \ddn(s) ) \cdot \f \Theta \dreidots \nabla \ddn(s) \o \ddn(s)  \right )  \\ 
&\quad + 2 \left (   (  \fn d(s) -  \ddn(s) ) \cdot \f \Theta \dreidots \partial_s( \nabla \ddn(s) \o \ddn(s))   \right ) :\left (   (  \fn d(s) -  \ddn(s) ) \cdot \f \Theta \dreidots \nabla \ddn(s) \o \ddn(s)  \right )  \,,
\end{align*}
where the second line can be estimated by 
\begin{align*}
 & \left (   (  \fn d(s) -  \ddn(s) ) \cdot \f \Theta \dreidots \partial_s( \nabla \ddn(s) \o \ddn(s))   \right ) :\left (   (  \fn d(s) -  \ddn(s) ) \cdot \f \Theta \dreidots \nabla \ddn(s) \o \ddn(s)  \right )  \\
 &\leq \|  \fn d(s) -  \ddn(s) \|_{\f L^6}^2 | \f \Theta |^2 \left ( \| \partial_s \nabla \ddn(s) \|_{\f L^3 } \| \ddn(s) \|_{\f L^\infty}+ \|  \nabla \ddn(s) \|_{\f L^3 } \|\partial_s  \ddn(s) \|_{\f L^\infty}\right )   \|  \nabla \ddn(s) \|_{\f L^3 } \| \ddn(s) \|_{\f L^\infty}\,,
\end{align*}
which proves the assertion. 

With respect to the integration-by-parts formula~\eqref{intH}, the fundamental theorem of calculus  grants that 
\begin{align*}
& \chi_{\|}  \left (\fn d(t) \cdot \f H , \ddn (t)\cdot \HH \right ) + \chi_{\bot} \left ( \fn d(t) \times \f H , \ddn (t)\times \HH \right ) -\chi_{\|}  \left (\fn d(0) \cdot \f H , \ddn (0)\cdot \HH \right ) - \chi_{\bot} \left ( \fn d(0) \times \f H , \dd (0)\times \HH \right ) \\
&  = \int_0^t \left [  \left ( \t \fn d , \chi_{\|}  \f H (\ddn \cdot \HH ) - \chi_{\bot} \f H \times ( \HH \times \ddn ) \right)  + \left ( \t \ddn ,\chi_{\|}  \HH ( \fn d \cdot \f H ) - \chi_{\bot}\HH \times ( \f H \times \fn d)   \right )    \right ]\de s \,.
\end{align*}
From a rearrangement, we may infer
\begin{align*}
& \chi_{\|}  \left (\fn d(t) \cdot \f H , \ddn (t)\cdot \HH \right ) + \chi_{\bot} \left ( \fn d(t) \times \f H , \ddn (t)\times \HH \right ) -\chi_{\|}  \left (\fn d(0) \cdot \f H , \ddn (0)\cdot \HH \right ) - \chi_{\bot} \left ( \fn d(0) \times \f H , \ddn (0)\times \HH \right ) \\
&  = \int_0^t \left [  \left ( \t \fn d , \chi_{\|}  \HH (\ddn \cdot \HH) - \chi_{\bot} \HH \times ( \HH \times \ddn ) \right)  + \left ( \t \ddn ,\chi_{\|}  \f H ( \fn d \cdot \f H ) - \chi_{\bot}  \f H \times ( \f H \times \fn d)   \right )    \right ]\de s \\
&\quad + \int_0^t   \left ( \t \fn d -\t \ddn  , \chi_{\|} ( \f H -  \HH )(\ddn \cdot \HH) - \chi_{\bot}( \f H -  \HH)  \times ( \HH \times \ddn ) \right)  \de s  \\
&\quad + \int_0^t  \left ( \t \ddn, \chi_{\|} (\HH - \f H)(\fn d \cdot \f H- \ddn \cdot \HH )   - \chi_{\bot}(\HH - \f H)\times (\f H \times \fn d- \HH \times \ddn)  \right )   \de s\,.
\end{align*}
Estimating the last line 
\begin{align*}
 &  \int_0^t  \left ( \t \ddn, -\chi_{\|} (\HH - \f H)(\fn d \cdot \f H- \ddn \cdot \HH )   - \chi_{\bot}(\HH - \f H)\times (\f H \times \fn d- \HH \times \ddn)  \right )   \de s\\
 &\leq  \int_0^t  \|  \t \ddn\|_{\f L^\infty} \left  (-\chi_{\|} \| \fn d \cdot \f H- \ddn \cdot \HH  \|_{\Le}^2  - \chi_{\bot} \| \f H \times \fn d- \HH \times \ddn \|^2_{\Le}   + c \| \HH - \f H\|_{\Le} ^2 \right )   \de s 
\end{align*}
proves the assertion.
\end{proof}

We are now ready to prove Theorem~\ref{thm:dis}.

\begin{proof}[Proof of Theorem~\ref{thm:dis}]
We split the proof in several steps. 
In the \textbf{first step}, we find
similar to~\cite[Corollary~5.1]{weakstrong} that
\begin{align*}
&\left ((\nabla \fn d (t)-\nabla \ddn(t)  )\o (\fn d(t) -  \ddn(t) ) \dreidotkom \f \Theta \dreidots \nabla \ddn(t)  \o \ddn(t) \right ) \\&\leq \frac{k}{2}\left \| \nabla \fn d(t) - \nabla \ddn(t) \right \|_{\Le}^2 + \frac{1}{2k}  \| (\fn d( t) - \ddn(t))\cdot \f \Theta \dreidots \nabla \ddn (t) \o \ddn (t)  \|_{\Le}^2 \\&\leq \frac{1}{4}\left ( \nabla \fn d(t) - \nabla \ddn(t) ; \f \Lambda :( \nabla \fn d(t) - \nabla \ddn(t))\right ) + \frac{1}{2k}  \| (\fn d( t) - \ddn(t))\cdot \f \Theta \dreidots \nabla \ddn (t) \o \ddn (t)  \|_{\Le}^2 \,.
\end{align*}
Here $k$ is the coercivity constant for the strongly elliptic tensor $\f\Lambda$ (see~\eqref{kill}). 
Additionally, we may infer from the integration-by-parts formulae~\eqref{intdn}  that
\begin{align*}
%\begin{split}
- &\left [\left ( \nabla \fn d   {}   ; \f \Lambda : \nabla \ddn  {}  \right ) + \left (  \nabla \fn d  {}  \o \fn d  {}   \dreidotkom \f \Theta \dreidots \nabla \ddn  {}   \o \ddn   {}    \right )  -  \chi_{\|}  \left (\f d  {}   \cdot \f H , \dd   {}  \cdot \HH \right ) - \chi_{\bot} \left ( \f d  {}   \times \f H , \dd   {}  \times \HH \right )\right ] \Big |_0^t 
% \\
\\
\leq {}& -  \int_0^t   \left [ (  \t \fn d  ,  \tq_n) + ( \fn q ,  \t \ddn)  \right ]    \de s 
% +\frac{2|\f \Theta|^2}{k}\| \nabla \dd \o \dd \|_{L^\infty( \f L^\infty)}^2\inttet{\left (  \t \dd - \t \f d  ,    \f d - \dd   \right )}
%\\
%&- 
%\left  ( \nabla \fn d(0); \f \Lambda : \nabla \ddn(0)\right ) - \left ( \nabla \fn d(0) \o \fn d(0) \dreidotkom \f \Theta \dreidots \nabla \ddn(0) \o \ddn(0)\right ) \\
%&
+\left ((\nabla \f d(0)-\nabla \dd (0))\o (\f d(0)-  \dd(0)) \dreidotkom \f \Theta \dreidots \nabla \dd (0)\o \dd(0)\right )\\
&+ \frac{1}{k} \left \| \left (  ( \fn d(0)- \ddn(0))\cdot \f \Theta \dreidots \nabla \ddn \o \ddn \right ): \f \Theta \dreidots \nabla \ddn \o \ddn     \right \|_{\Le}^2
%&+\chi_{\|}  \left (\f d(0) \cdot \f H , \dd (0)\cdot \HH \right ) + \chi_{\bot} \left ( \f d(0) \times \f H , \dd (0)\times \HH \right ) 
%\\
+  \int_0^t \mathcal{K}\mathcal{E} 	(\fn v ,\fn d ,\f H  | \vvn , \ddn , \HH ) \de s \\
%& + \int_0^t    \left ( \t \f d -\t \dd  , \frac{1}{k } \left ( ( \fn d - \ddn) )\cdot \f\Theta \dreidots \nabla \ddn \o \ddn \right ) : \f \Theta \dreidots \nabla \ddn \o \ddn \right)\de s \\ 
%& + \int_0^t    \left ( \t \f d -\t \dd  ,   \chi_{\|} ( \f H -  \HH )(\dd \cdot \HH) - \chi_{\bot}( \f H -  \HH)  \times ( \HH \times \dd ) \right)\de s \\ 
& + \int_0^t    \left [\left ( \t \f d -\t \dd  ,\f a_{n} \right)+   c \|  \t \dd\|_{\f L^\infty}  \| \HH - \f H\|_{\Le} ^2 +\langle (R_n-I) \t \dd ,  \f q( \fn d)\rangle  \right ] \de s 
  + \frac{1}{2}  \mathcal{E}(\fn v ,\fn d, 0| \vvn , \ddn, 0   )(t) \,.
%   \\
%   & + \inttet{\ \left ( \nabla( R_n -I)\t \dd ; \f \Lambda : \nabla \fn d + \fn d \cdot \f \Theta \dreidots \nabla \fn d \o \fn d \right ) }\\
% & + \inttet{ \left ( (R_n- I) \t \dd , \nabla \fn d  : \f \Theta \dreidots \nabla \fn d - \chi_{\|} \f H ( \fn d \cdot \f H ) + \chi_{\bot} \f H\times ( \f H \times \fn d) \right )}
%\end{split}
\numberthis\label{auscorn}
\end{align*}
holds for all $t\in[0,T]$. 
We used the abbreviation~\eqref{tddqn}. 
Note that in the definition of $\fn q$, the projection $R_n$ appears. 
Therefore, the term $( \fn q ,\t \dd)$ is simultaneously added and subtracted. This leads to the second term appearing on the right-hand side of~\eqref{auscorn} and the second to the last term on the right-hand side of~\eqref{auscorn} incorporating the difference of the projection and the identity, \textit{i.e.}, $(R_n-I)$. 
This projection may be inserted in the case of $\tq_n$ since $\t \fn d $ is an element of the appropriate subspace. 

In the \textbf{second step}, we recall the shifted energy inequality for the test functions~$(\vvn, \ddn)$ (see~\cite{diss}). 
We add and simultaneously subtract the equations~\eqref{vdis} and equation~\eqref{ddis} evaluated at $( \vvn ,\ddn)$  tested with $\vvn$  and  $\tq_n$, respectively. This leads to
 \begin{align*}
%\begin{split}
 &\frac{1}{2}\|\vvn (t)\|_{\Le}^2 + \mathcal{F}(\ddn(t))   + \inttet{\left [(\mu_1+\lambda^2)\|\ddn\cdot \syvn\ddn\|_{L^2}^2 +
  \mu_4 \|\syvn\|_{\Le}^2 \right ]}  \\
& +\inttet{\left [( \mu_5+\mu_6-\lambda^2)\|\syvn\ddn\|_{\Le}^2  +  \|\ddn \times  \tq_n\|_{\Le}^2 - \langle(I-R_n)\t \dd , \f q( \dd)\rangle   \right ]}
\\
& \qquad =  \left ( \frac{1}{2}\|\vvn(0) \|_{\Le}^2 + \F( \ddn(0))\right )
 + \inttet{\left [\langle \f g(s) , \vvn(s) \rangle + \left (\tilde{\mathcal{A}}(\vvn,\ddn), \begin{pmatrix}
 \vvn (s)\\   \tq_n(s)
 \end{pmatrix}\right )
% +\frac{1}{2} \left ( \t |\ddn|^2 + ( \vvn\cdot \nabla ) |\ddn|^2, \ddn \cdot \tq_n\right ) 
 %+(  ( \mu_2+ \mu_3) - \lambda ) \left ( \dd \times \tq , \dd \times \syv \dd \right ) 
 \right ]}\, .
%\end{split}\label{shiftenergyn}
\end{align*}
Note that we have to add and subtract the variational derivative of $\tq$ to be able to use the chain rule to get $\mathcal{F}(\tilde{\f d(t)})$ on the left-hand side. This leads to the last term on the left-hand side. 
Here $\tilde{\mathcal{A}}_n$ is given by
\begin{align}
\tilde{\mathcal{A}}_n(\vvn,\ddn) = \begin{pmatrix}
\t \vvn + ( \vvn \cdot \nabla )\vvn - \nabla \ddn^T ( | \ddn |^2 - \ddn \o \ddn ) \tq_n - \di \tilde{\f T}^L_n- \f g \\
\t \ddn +  \left ( | \ddn|^2 I - \ddn \o \ddn\right ) \left ( ( \vvn \cdot \nabla ) \ddn - \skvn \ddn + \lambda \syvn \ddn + \tq _n \right )
\end{pmatrix}\,.\label{ansch}
\end{align}
With the identity 
\begin{align}
\t \ddn = ( 1 - |\ddn|^2 ) \t \ddn + \frac{ 1}{2}\t | \ddn|^2 \ddn - \ddn \times\ddn \times \t \ddn \,\label{tdumform},
\end{align}
we find
 \begin{align*}
%\begin{split}
 &\frac{1}{2}\|\vvn (t)\|_{\Le}^2 + \mathcal{F}(\ddn(t))   + \inttet{\left [(\mu_1+\lambda^2)\|\ddn\cdot \syvn\ddn\|_{L^2}^2 +
  \mu_4 \|\syvn\|_{\Le}^2 \right ]}  \\
& +\inttet{\left [( \mu_5+\mu_6-\lambda^2)\|\syvn\ddn\|_{\Le}^2  +  \|\ddn \times  \tq_n\|_{\Le}^2\right ]}
\\
& \qquad =  \left ( \frac{1}{2}\|\vvn(0) \|_{\Le}^2 + \F( \ddn(0))\right )
 + \inttet{\left [\langle \f g , \vvn \rangle + \left (\mathcal{A}_n(\vvn,\ddn), \begin{pmatrix}
 \vvn \\   \ddn \times \tq_n
 \end{pmatrix}\right )
% +\frac{1}{2} \left ( \t |\ddn|^2 + ( \vvn\cdot \nabla ) |\ddn|^2, \ddn \cdot \tq_n\right ) 
 %+(  ( \mu_2+ \mu_3) - \lambda ) \left ( \dd \times \tq , \dd \times \syv \dd \right ) 
 \right ]}
\\
&\qquad + \inttett{\left (( 1 - |\ddn|^2 ) \t \ddn + \frac{ 1}{2}\t | \ddn|^2 \ddn, \tq_n\right )+ \langle(I-R_n)\t \dd , \f q( \dd)\rangle   } 
 \, .
%\end{split}
\numberthis\label{shiftenergyn}
\end{align*}
Here $\mathcal{A}_n$ is given by~\eqref{An}.

In the same way by adding and subtracting equation~\eqref{vdis} evaluated at $\vvn$ and tested with $\fn v$ and equation~\eqref{ddis} evaluated at $\ddn$ and tested with $ \fn q$, we obtain that
\begin{align*}
%\begin{split}
-{}& \inttett{\left ( \t \vvn ,\fn v \right ) + \left (  \t \ddn ,  \fn q\right ) + \mu_4 \left ( \syvn, \syn v \right ) + \left (\ddn \times \tq_n , \ddn  \times \fn q \right ) }\\
={}& \inttett{\left ( ( \vvn \cdot \nabla ) \vvn , \fn v  \right )+ (\mu_1+\lambda^2) \left (\ddn \cdot \syvn \ddn , \ddn \cdot \syn v \ddn \right ) - \langle \f g , \fn v \rangle }
\\
+{}& \inttett{
 ( \mu_5+ \mu_6-\lambda^2) \left ( \syvn \ddn , \syn v \ddn \right ) -\lambda \left (\ddn\times \tq_n ,\dd \times  \syn v \ddn \right )}\\
&-\inttett{
%+\lambda( \mu_2+ \mu_3 ) \left (\dd\times \syv \dd  ,\dd \times  \sy v \dd \right )
 \left (\ddn \times \tq_n , \ddn \times \skn v \ddn\right )+\left ( \ddn \times \left (( \vvn \cdot \nabla) \ddn -\skvn \ddn\right ) ,  \ddn \times  \fn q \right )
}\\
&- \inttett{  \left ( \lambda {\ddn}\times  \syvn \ddn,  \ddn \times   \fn q \right ) + \left (\nabla \ddn ^T \left (| \ddn |^2 I -\ddn \o \ddn \right )\tq_n,  \fn v \right )}\\
& + \inttett{\left (( 1 - |\ddn|^2 ) \t \ddn + \frac{ 1}{2}\t | \ddn|^2 \ddn, \fn q \right )-\left (\mathcal{A}_n(\vvn,\ddn), \begin{pmatrix}
\fn v \\   \ddn \times \fn q 
\end{pmatrix} \right ) } 
\,.
%\end{split}
\numberthis\label{strongeqn}
\end{align*}
Note that equation~\eqref{strongeqn} is valid for all test functions $(\fn v , \fn q)$ since basically it is just zero. 
Again, we employ~\eqref{tdumform} to replace~$\tilde{\mathcal{A}}_n$ by the last line of~\eqref{strongeqn}.

In the \textbf{third step}, we want to derive a similar equation for the solutions $(\fn v, \fn d)$ of the approximate system~\eqref{eq:dis}. 
One is tempted to test equation~\eqref{vdis} with $\vvn$ and equation~\eqref{ddis} with~$\tq_n$. But the equation~\eqref{vdis} does not hold for this test function. Therefore, we test the equations with the associated projected value, \textit{i.e.}, equation~\eqref{vdis} tested with $P_n \vvn$ and add and subtract the equation~\eqref{vdis} tested with~$\vv$ simultaneously. 
 This leads to
\begin{align*}
%\begin{split}
-{}& \inttett{\left ( \t \fn v ,\vvn \right ) + \left (  \t \fn d  ,  \tq_n\right ) + \mu_4 \left ( \syn v, \syvn \right ) + \left (\fn d \times  \fn q  , \fn d  \times  \tq_n \right ) }\notag\\
={}& \inttett{\left ( ( \fn v \cdot \nabla ) \fn v , \vvn  \right )+ (\mu_1+\lambda^2) \left (\fn d \cdot \syn v \ddn , \fn d \cdot \syvn  \fn d \right )- \langle \f g , \fn v \rangle  }
\\
+{}& \inttett{
 ( \mu_5+ \mu_6-\lambda^2) \left ( \syn v \fn d , \syvn \fn d \right ) -\lambda \left (\fn d\times  \fn q ,\fn d \times  \syvn \fn d \right )}\\
&-\inttett{
 \left (\fn d \times  \fn q  , \fn d \times \skvn \fn d\right )+\left ( \fn  d\times \left (( \fn v   \cdot \nabla) \fn d -\skn v \fn d \right ) ,  \fn d  \times   \tq_n \right )
}\\
&- \inttett{  \left (  \lambda {\fn  d }\times  \syn v \fn d ,  {\fn d}\times      \tq_n \right ) + \left (\nabla \fn d  ^T \left (| \fn d |^2 I -\fn d   \o \fn d  \right ) \fn q,  \vvn \right )}\\
&+ \int_0^t \left ( \mathcal{A}_n ( \fn v ,\fn d) , \begin{pmatrix}
P_n \vvn -\vvn \\ 0 
\end{pmatrix}\right ) \de s 
\,.
%\end{split}
\numberthis\label{weakeqn}
\end{align*}

In the \textbf{fourth step}, 
the results of the previous steps are put together to get the relative energy inequality. 
%
%we estimate the relative energy and the relative dissipation appropriately similar to the proof of~\cite[Corollary~6.1]{weakstrong}.  
Calculating the relative energy~$\mathcal{E} $ and the relative dissipation $\mathcal{W}$, inserting the energy equality of Proposition~\ref{prop:enrgydis} and the shifted energy equality~\eqref{shiftenergyn} yields
\begin{align*}
%\begin{split}
\mathcal{E}&(\fn v ,\fn d , \f H | \vvn , \ddn, \HH  )(t) + \int_0^t\mathcal{W}(\fn v, \fn d | \vvn,\ddn ) \de s \leq{} 
\mathcal{D}_0(\fn v ,\fn d, \f H  | \vvn , \ddn , \HH  )(0)
 \\& - 2\int_0^t  \left [  (\mu_1 +\lambda^2 )( \fn d \cdot \syn v \fn d , \ddn\cdot \syvn  \ddn )+( \mu_5 + \mu_6 - \lambda ^2)( \syn v \fn d , \syvn \ddn ) \right ]    \de s  
  \\& 
  - 2 \int _0^t \left [ \mu_4  ( \syn v ; \syvn ) +( \fn d\times \fn q , \ddn \times  \tq_n )\right ] \de s  
 + \int_0^t \langle \f g , \f v + \vvn \rangle \de s \\& 
  -  \int_0^t\left [( \fn v  , \t \vvn  ) + ( \t \fn v  , \vvn) \right ]\de s -  \int_0^t   \left [ ( \t \fn d  ,   \tq_n) + (  \fn q ,
\t \ddn)  \right ]    \de s+ \frac{1}{2} \mathcal{E}(\fn v ,\fn d , 0| \vvn , \ddn, 0  )(t) 
\\&
+ \inttett{\mathcal{K}  \mathcal{E}(\fn v ,\fn d , \f H  | \vvn , \ddn , \HH ) 
% + \frac{1}{k} \left \| \left (  ( \fn d(0)- \ddn(0))\cdot \f \Theta \dreidots \nabla \ddn \o \ddn \right ): \f \Theta \dreidots \nabla \ddn \o \ddn     \right \|_{\Le}^2
% \\
%&
%+\left ((\nabla \fn d(0)-\nabla \ddn (0))\o (\fn d(0)-  \ddn(0))\dreidotkom \f \Theta \dreidots \nabla \ddn (0)\o \ddn(0)\right )   
+ \left (\mathcal{A}_n(\vvn,\ddn), \begin{pmatrix}
 \vvn \\ 
  \ddn \times \tq_n
 \end{pmatrix}\right )+ \langle(I-R_n)\t \dd , \f q(\dd) - \f q( \fn d)\rangle   }
\\
%& + \int_0^t   \left [ \left ( \t \f d -\t \dd  , \chi_{\|} ( \f H -  \HH )(\dd \cdot \HH) - \chi_{\bot}( \f H -  \HH)  \times ( \HH \times \dd ) \right)  -  \|  \t \dd\|_{\f L^\infty} c \| \HH - \f H\|_{\Le} ^2\right ]   \de s  
& + \int_0^t   \left [ \left ( \t \f d -\t \dd  ,\f a_{n} \right)  - c \|  \t \dd\|_{\f L^\infty}  \| \HH - \f H\|_{\Le} ^2+\left (( 1 - |\ddn|^2 ) \t \ddn + \frac{ 1}{2}\t | \ddn|^2 \ddn, \tq_n\right )\right ]   \de s  
%\\
%& + \int_0^t    \left ( \t \f d -\t \dd  , \frac{1}{k } \left ( ( \fn d - \ddn) )\cdot \f\Theta \dreidots \nabla \ddn \o \ddn \right ) : \f \Theta \dreidots \nabla \ddn \o \ddn \right)\de s 
 \, ,
%\end{split}
\numberthis\label{vormeinsn}
\end{align*}
where $\mathcal{D}_0$ is given in~\eqref{D0}. Note also the definition~\eqref{tddqn}. 
Inserting now equation~\eqref{strongeqn} and~\eqref{weakeqn},
 respectively,  yields
\begin{align*}
%\begin{split}
\frac{1}{2}\mathcal{E}&(\fn v ,\fn d , \f  H | \vvn , \ddn , \HH )(t) + \int_0^t\mathcal{W}(\fn v, \fn d | \vvn,\ddn )  \de s 
\\
\leq{}&\mathcal{D}_0(\fn v ,\fn d, \f H  | \vvn , \ddn , \HH  )(0)
% + \delta\int_0^t\mathcal{W}(s) \de s
  +  \int_0^t\mathcal{K} \mathcal{E}(\fn v ,\fn d, \f H  | \vvn , \ddn, \HH )\de s + \int_0^t \left [( ( \vvn \cdot \nabla ) \vvn , \fn v ) + ( ( \fn v \cdot\nabla ) \fn v , \vvn ) \right ]\de s \\
& + (\mu_1 +\lambda^2)\int_0^t    \left ( \ddn  \cdot \syvn \ddn  , \syn v : \left ( \ddn  \o \ddn  - \fn d \o \fn d   \right )   \right )  \de s \\
&+ (\mu_1 +\lambda^2)\int_0^t    \left ( \fn d  \cdot \syn v \fn d  , \syvn : \left ( \fn d  \o \fn d  - \ddn  \o \ddn    \right )   \right )  \de s \\
& + ( \mu_5 + \mu_6  -\lambda^2) \int_0^t  \left [\left ( \syvn \ddn  , \syn v ( \ddn  - \fn d )\right  )  + \left (    \syn v \fn d , \syvn ( \fn d - \ddn  ) \right ) \right ]   \de s \\
&
-\inttett{
2 \left ( \fn d \times  \fn q, \ddn \times \tq_n  \right ) - \left (  \ddn \times \tq_n, \ddn \times \fn q\right ) - \left ( \fn d \times    \fn q , \fn d \times  \tq_n\right )
}
\\
& -\inttet{\left [ \left (  \ddn \times \skvn \ddn , \ddn \times   \fn q \right )- \left (  \fn d \times \skvn \fn  d,\fn d \times   \fn q \right )   \right ]}\\
&-\inttet{\left [  \left (  \fn d \times \skn v \fn d , \fn d \times   \tq_n \right )-\left (\ddn \times  \skn v \ddn   , \ddn \times \tq_n\right )  \right ]}\\
& -  \lambda \int_0^t \left [\left ( \ddn \times \syn v \ddn  , \ddn  \times \tq_n \right )+ \left ( \fn d \times \syvn \fn d  , \fn d \times  \fn q \right )\right ]  \de s
\\
& +  \lambda \int_0^t \left [\left ( \ddn \times \syvn \ddn  , \ddn  \times \fn q  \right )+ \left ( \fn d \times \syn v \fn d  , \fn d \times   \tq_n \right )\right ]  \de s
 \\
&+ \int_0^t \left [(\ddn \times ( \vvn \cdot \nabla ) \ddn  , \ddn \times  \fn q) + (\fn d \times ( \fn v \cdot \nabla )\fn d , \fn d \times   \tilde{\f q}_n)
\right ]\de s  
 \\
&- \int_0^t \left [
 ( \nabla \ddn ^T(| \ddn |^2 I -\ddn \o \ddn) \tq_n ,  \fn v )  +  ( \nabla \fn d ^T(|\fn d|^2 I - \fn d \o \fn d)   \fn q ,  \vvn )
\right ]\de s  
\\
%&
%+\left ((\nabla \fn d(0)-\nabla \ddn  (0))\o (\fn d(0)-  \ddn (0))\dreidotkom \f \Theta \dreidots \nabla \ddn  (0)\o \ddn (0)\right ) \\
%& + \int_0^t   \left [ \left ( \t \f d -\t \dd  , \chi_{\|} ( \f H -  \HH )(\dd \cdot \HH) - \chi_{\bot}( \f H -  \HH)  \times ( \HH \times \dd ) \right)  -  \|  \t \dd\|_{\f L^\infty} c \| \HH - \f H\|_{\Le} ^2\right ]   \de s  
& +  \inttet{  \left [ \left ( \t \f d -\t \dd  , \f a_{n} \right)    +\left (( 1 - |\ddn|^2 ) \t \ddn + \frac{ 1}{2}\t | \ddn|^2 \ddn, \tq_n-\fn q \right )+ \langle(I-R_n)\t \dd , \f q(\dd) - \f q( \fn d)\rangle \right ]  }  
\\
%&+ \frac{1}{k} \left \| \left (  ( \fn d(0)- \ddn(0))\cdot \f \Theta \dreidots \nabla \ddn \o \ddn \right ): \f \Theta \dreidots \nabla \ddn \o \ddn     \right \|_{\Le}^2
%\\
& 
 + \inttet{ \left [\left (\mathcal{A}_n (\vv,\dd  ), \begin{pmatrix}
 \vvn- \fn v  \\ \ddn\times (\tq - \fn q)  
 \end{pmatrix}\right )- c  \|  \t \dd\|_{\f L^\infty}  \| \HH - \f H\|_{\Le} ^2+ \left (\mathcal{A}_n(\fn v ,\fn d) , \begin{pmatrix}
 P_n \vvn-\vvn\\0
 \end{pmatrix}\right )\right ]}
\\
    ={}& \sum_{i=1}^3 I_i+ (\mu_1+\lambda^2) (I_4+I_5) + ( \mu_5+\mu_6-\lambda^2) I_6-\sum_{1=7}^9 I_i - \lambda I_{10} +\lambda I_{11} +I_{12}- I_{13} + \sum_{i=14}^{15} I_i \,.
%\end{split}
\numberthis \label{zurueck}
\end{align*}
The remaining part of the proof consists of appropriately estimating the right hand side of the foregoing inequality similar to the proof of~\cite[Corollary~6.1]{weakstrong}.  

In \textbf{step five}, the different dissipative terms, \textit{i.e.},the terms $I_3$--$I_{13}$, are estimated.
The terms $I_3$--$I_6$ are already estimated in~\cite{weakstrong}, this implies that
\begin{align*}
I_3+ (\mu_1+\lambda^2)(I_4+I_5) + ( \mu_5+\mu_6-\lambda^2) I_6  \leq  \delta \inttet{\mathcal{W}( \fn v, \fn d | \vvn , \ddn) } + \inttet{\mathcal{K}\mathcal{E}(\fn v ,\fn d ,0| \vvn , \ddn ,0 )}\,.
\end{align*}

The terms $I_7$--$I_{13}$ are estimated similar. We exemplify the calculations for the term $I_{12}-I_{13}$. Using the properties of the cross product (see Section~\ref{sec:not}),  we observe for the term $I_{12}-I_{13}$ that  
\begin{align*}
I_{12}-I_{13} ={}& \int_0^t \left [(\ddn \times ( \vvn \cdot \nabla ) \ddn  , \ddn \times  (\fn q- \tq_n)) + (\fn d \times ( (\fn v- \vvn) \cdot \nabla )\fn d , \fn d \times   \tilde{\f q}_n)
\right ]\de s  
 \\
&+ \int_0^t \left [
 ( \ddn \times ((\vvn - \fn v ) \cdot \nabla )\ddn , \ddn \times \tq_n  )  -(\fn d \times (\vvn \cdot \nabla) \fn d ,\fn d\times  (  \fn q- \tq_n)   )
\right ]\de s  
\\
 =&\int_0^t(\ddn \times ( \vvn \cdot \nabla ) \ddn  , (\ddn-\fn d)  \times  (\fn q- \tq_n)) \de s  
+ \int_0^t (\ddn  \times ( (\fn v- \vvn) \cdot \nabla ) \ddn  , (\fn d- \ddn ) \times   \tilde{\f q}_n)
\de s  
 \\
% &+ \int_0^t \left [
%   ((\fn d- \ddn ) \times ( (\fn v- \vvn) \cdot \nabla ) (\fn d-\ddn)   , \fn d \times   \tilde{\f q}_n)
%\right ]\de s
% \\
 &+ \int_0^t \left [
   ((\fn  d- \ddn ) \times ( (\fn v- \vvn) \cdot \nabla ) \ddn  +  \fn d   \times ( (\fn v- \vvn) \cdot \nabla ) (\fn d-\ddn ), \fn d \times  \tilde{\f q}_n)
\right ]\de s
 \\
&+ \int_0^t \left [
 ((\ddn-\fn d) \times ( \vvn \cdot \nabla ) \ddn +\fn d  \times ( \vvn \cdot \nabla ) (\ddn-\fn d)  ,\fn d\times    \fn q-\ddn \times  \tq_n + (\ddn - \fn d )\times \tq_n)   )
\right ]\de s  
%\\
%&+ \int_0^t \left [
% ( (\ddn - \fn d )\times (\vvn \cdot \nabla) (\fn d-\ddn) ,\fn d\times    \fn q-\ddn \times  \tq_n + (\ddn - \fn d )\times \tq_n)   )
%%
%\right ]\de s 
 \,.
\end{align*}
We keep the first term on the right-hand side of the second equality and estimate the remaining terms
\begin{align*}
I_{12}-I_{13} \leq {} &\int_0^t(\ddn \times ( \vvn \cdot \nabla ) \ddn  , (\ddn-\fn d)  \times  (\fn q- \tq_n)) \de s  
 \\ & + \delta \inttet{\| \fn v - \vvn \|_{\f L^6}^2 }+ C_{\delta} \| \ddn \|_{L^\infty(\f L^\infty)}^2\inttet{ \| \tq_n \|_{\f L^3}^2  \| \nabla \ddn \|_{\f L^3}^2\| \fn d -\ddn \|_{\f L^6}^2  }\\ \
% & + C_{\delta} \| \ddn \|_{L^\infty(\f L^\infty)}^2\inttet{ \| \tq_n \|_{\f L^3}^2   \| \ddn \|_{\f L^\infty}\| \nabla \fn d -\nabla \ddn | \fn d -\ddn |\|_{\Le}^2 }\\
& + C_{\delta} \| \fn d \|_{L^\infty(\f L^\infty)}^2\inttet{ \| \tq_n \|_{\f L^3}^2  \left (\| \nabla \ddn \|_{\f L^3}^2\| \fn d -\ddn \|_{\f L^6}^2 + \| \fn d \|_{L^\infty(\f L^\infty)}^2\| \nabla \fn d -\nabla \ddn \|_{\Le}^2  \right ) }\\
&+ \delta \inttet{\|\fn d \times \fn q -\ddn \times \tq_n\|_{\Le}^2 } 
\\
&+ C_\delta \inttet{\| \vvn \|_{\f L^\infty}^2\left ( \|\nabla  \ddn\|_{\f L^3}^2 \| \fn d - \ddn\|_{\f L^6}^2 + \|\fn d\|_{L^\infty(\f L^\infty)}^2 \| \nabla \fn d - \nabla \ddn\|_{\Le}^2 \right ) }\\
%& + \inttet{\| \vvn \|_{\f L^\infty}\| \tq_n \|_{\f L^3}  }
%\\
& + \inttet{\| \vvn \|_{\f L^\infty}\| \tq_n \|_{\f L^3} \left (  \|\nabla  \ddn\|_{\f L^3} \| \fn d - \ddn\|_{\f L^6}^2+ \| \fn d\|_{\f L^\infty} \| \nabla \fn d -\nabla \ddn \|_{\Le} \| \fn d - \ddn \|_{\f L^6} \right )}\,.
\end{align*}
Similar estimates for the terms $I_7$--$I_{11}$ let us conclude that
\begin{multline*}
-\sum_{1=7}^9 I_i - \lambda I_{10} +\lambda I_{11} +I_{12}- I_{13} \leq  \delta \inttet{\mathcal{W}(\fn v, \fn d | \vvn, \ddn )} + \inttet{\mathcal{K}\mathcal{E}(\fn v ,\fn d , \f H | \vvn , \ddn , \HH)}
\\
\int_0^t(\ddn \times \left (( \vvn \cdot \nabla ) \ddn - \skvn \ddn +\lambda \syvn \ddn + \tq_n\right )  , (\ddn-\fn d)  \times  (\fn q- \tq_n)) \de s   
 \,.
\end{multline*}
Adding and subtracting the term~$\inttet{\ddn \times \t \ddn , ( \ddn -\fn d ) \times (\fn q - \tq_n)} $ leads to 
\begin{multline*}
-\sum_{1=7}^9 I_i - \lambda I_{10} +\lambda I_{11} +I_{12}- I_{13} \leq \delta \inttet{\mathcal{W}(\fn v, \fn d | \vvn, \ddn )} + \inttet{\mathcal{K}\mathcal{E}(\fn v ,\fn d, \f H  | \vvn , \ddn, \HH  )}\\
- \int_0^t(\ddn \times \t \ddn   , (\ddn-\fn d)  \times  (\fn q- \tq_n)) \de s   
+   \inttet{\left (\mathcal{A}_n (\vv, \dd ),\begin{pmatrix}
0 \\  (\ddn-\fn d)  \times  (\fn q- \tq_n)
\end{pmatrix}\right ) }\,.
\end{multline*}

It remains to estimate the first term in the second line of the previous inequality.
It results in the \textbf{sixth step}, where the difference of the variational derivatives are estimated.
First we observe that the definition~\eqref{qn} of $\fn q$ and $\tq_n$  incorporates the projections $R_n$. 
Since $R_n$ is an $\Le$-projection, we find with the definition~\eqref{qdef} that
\begin{align*}
&\int_0^t(\ddn \times \t \ddn   , (\ddn-\fn d)  \times  (\fn q- \tq_n)) \de s    \\
&= - \int_0^t((\ddn-\fn d)  \times   \ddn \times \t \ddn   , \fn q- \tq_n) \de s   \\
&=  \int_0^t\left (R_n \left ((\ddn-\fn d)  \times   \ddn \times \t \ddn\right )    ,  \Lap ( \fn d -\ddn) \right ) \de s   \\ 
&\quad + \int_0^t\left (R_n \left ((\ddn-\fn d)  \times   \ddn \times \t \ddn\right )    ,  \di \left ( \fn d \cdot \f \Theta \dreidots \nabla \fn d \o \fn d - \ddn \cdot \dreidots \f \Theta \dreidots \nabla \ddn \o \ddn \right )  \right )\de s   \\
&\quad - \int_0^t\left (R_n \left ((\ddn-\fn d)  \times   \ddn \times \t \ddn\right )    , \nabla  \fn d : \f \Theta \dreidots \nabla \fn d \o \fn d -\nabla \ddn : \dreidots \f \Theta \dreidots \nabla \ddn \o \ddn \  \right )\de s   \\
&\quad + \int_0^t\left (R_n \left ((\ddn-\fn d)  \times   \ddn \times \t \ddn\right )    ,  \chi_{ \|} \left ( ( \fn d \cdot \f H ) \f H - ( \ddn \cdot \HH ) \HH \right ) - \chi_{\bot} \left ( \f H \times ( \f H \times \fn d ) - \HH \times ( \HH \times \ddn)  \right )  \right )\de s   \,.\numberthis \label{nummer}
\end{align*}
Rearranging the right-hand side leads to
   \begin{align*}
     \fn d \cdot \f \Theta \dreidots \nabla \fn d \o \fn d - \ddn \cdot \f \Theta 		\dreidots \nabla \ddn \o \ddn 
  & ={}   \fn d \cdot \f \Theta \dreidots \left (\nabla \fn d \o \fn d- \nabla \ddn \o \ddn\right )
%  \\ & \qquad \quad+  \ddn \cdot \f \Theta \dreidots \left (\nabla \fn d \o \fn d- \nabla \ddn \o \ddn\right ) 
  +(\fn d-\ddn) \cdot \f \Theta \dreidots \nabla \ddn \o \ddn
  \intertext{and }
   \nabla \fn d : \f \Theta \dreidots \nabla \fn d \o \fn d - \nabla \ddn : \f \Theta \dreidots \nabla \ddn \o \ddn 
   & ={} (\nabla \fn d-\nabla \ddn) : \f \Theta \dreidots \left (\nabla \fn d \o \fn d - \nabla \ddn \o \ddn\right ) \\&  \quad +\nabla \ddn : \f \Theta \dreidots \left (\nabla \fn d \o \fn d - \nabla \ddn \o \ddn\right ) +(\nabla \fn d-\nabla \ddn) : \f \Theta \dreidots \nabla \ddn \o \ddn\,
   \end{align*}
   as well as 
   \begin{align*}
   & \chi_{ \|} \left ( ( \fn d \cdot \f H ) \f H - ( \ddn \cdot \HH ) \HH \right ) - \chi_{\bot} \left ( \f H \times ( \f H \times \fn d ) - \HH \times ( \HH \times \ddn)  \right ) \\
    &\qquad= \chi_{ \|} \left ( (\fn d \cdot \f H - \ddn \cdot \HH ) \f H + ( \ddn \cdot \HH ) ( \f H - \HH ) \right ) - \chi _{ \bot} \left (\f H \times ( \f H \times \fn d - \HH \times \ddn ) + ( \f H - \HH ) \times ( \HH \times \ddn) \right ) \,. 
   \end{align*}
Note that the $\Le$-projection $R_n$ is stable in $\He$ and $\f L^\infty$. By a simple interpolation argument between $\f L^\infty$ and $\f L^2$, the stability of the $\Le$-projection in $\f L^6$ follows immediately.  
Performing an integration-by-parts on the terms in the first two lines on the right-hand side of~\eqref{nummer} and estimating yields
\begin{align*}
&\int_0^t(\ddn \times \t \ddn   , (\ddn-\fn d)  \times  (\fn q- \tq_n)) \de s    \\ 
& \leq   \inttet{ \|( \ddn - \fn d) \times (\ddn  \times  \t \ddn) \|_{\Hb } | \f \Lambda |  \| \nabla \fn d - \nabla \ddn\|_{\Le}}
\\
&\quad + \inttet{\| ( \ddn - \fn d) \times( \ddn \times \t \ddn) \|_{\Hb} \| \fn d \|_{\f L^\infty}  \left  \| \f\Theta \dreidots \left ( \nabla \fn d \o \fn d - \nabla \ddn \o \ddn \right ) \right \| _{\Le}  }\\
&\quad + \inttet{\| ( \ddn - \fn d) \times (\ddn \times \t \ddn) \|_{\Hb}  \| \fn d -\ddn\|_{\f L^6} | \f \Theta |  \| \nabla \ddn \|_{\f L^3}\| \ddn \|_{\f L^\infty}  }\\
&\quad+ \inttet{\| ( \ddn - \fn d) \times (\ddn \times \t \ddn) \|_{\f L^\infty } \| \nabla \fn d -\nabla \ddn \|_{\Le}\left  \| \f \Theta\dreidots \left ( \nabla\fn d\o \fn d -\nabla \ddn \o \ddn \right ) \right \| _{\Le} } \\
&\quad + \inttet{\| \ddn \times\t \ddn \|_{\f L^\infty}\| \fn d - \ddn \|_{\f L^6} \| \nabla \ddn \|_{\f L^3}   \left  \| \f \Theta\dreidots \left ( \nabla\fn d\o \fn d -\nabla \ddn \o \ddn \right ) \right \| _{\Le} }
\\&\quad + \inttet{\| \ddn \times\t \ddn \|_{\f L^\infty}\| \fn d - \ddn \|_{\f L^6}    \left  \|  \nabla\fn d -\nabla \ddn \right \| _{\Le} | \f\Theta| \| \nabla \ddn \|_{\f L^3} \| \dd \|_{ \f L^\infty}}
\\&
\quad + \inttet{\| \ddn \times\t \ddn \|_{\f L^\infty}  \| \fn d - \ddn \|_{\f L^6}  \| \f H \|_ { \f L^3} \left ( \chi_{\|} \| \fn d \cdot \f H - \dd \cdot \HH \|_{ \Le} +\chi_{\bot}\| \f H \times \fn d - \HH \times \ddn\|_{ \Le} \right )    }
\\&\quad + \inttet{\| \ddn \times\t \ddn \|_{\f L^\infty}  \| \fn d - \ddn \|_{\f L^6}  \| \f H - \HH \|_{ \f L^2} \left ( \chi_{\|} \| \dd \cdot \HH \|_{ \f L^3} + \chi_{\bot} \| \HH \times \dd\|_{ \f L^3} \right )  }
\,,
\end{align*}
from where we may conclude with the algebraic relation in~\cite[Proposition~A.1]{weakstrong} 
%and Corollary~\ref{cor:H1} 
 that
\begin{align*}
&\int_0^t(\ddn \times \t \ddn   , (\ddn-\fn d)  \times  (\fn q- \tq_n)) \de s \leq \inttet{\mathcal{K} \left (\mathcal{E}(\fn v ,\fn d,\f H | \vv ,\dd ,\HH)+ \| \f H - \HH\|_{ \Le}^2\right )}\,.
\end{align*}

The \textbf{seventh step} consists of estimating the contribution due to the non-convex character of the considered potential. 
To rearrange the term $I_{14}$, we use equation~\eqref{ddis} tested with $R_n \f a_{n} $ and add as well as subtract equation~\eqref{ddis} tested with $\f a_{n}$. Additionally, we add and subtract the term $\inttet{(\mathcal{A}_n(\vvn,\ddn) , (0,\fn d \times \f a_{n})^T)}$, with $\f a_{n} $ given in~\eqref{anm}. 
This yields
\begin{align*}
%\begin{split}
&\inttet{\left (\t \ddn - \t \fn d , 
%\fn d -\ddn
\f a_{n} \right ) }{}
\\
&=\inttet{\left (  \fn d \times \f a_{n} , \ddn \times (\vvn \cdot \nabla) \ddn - \fn d \times ( \fn v \cdot \nabla ) \fn d 
\right )}
-\inttet{\left (  \fn d \times \f a_{n} , \ddn \times\skvn \ddn - \fn d \times \skn v\fn d  
\right )}
\\
&\quad+ \inttet{\left (  \fn d \times \f a_{n} ,  \lambda \left ( \ddn \times \syvn \ddn -\fn d \times  \syn v \fn d \right ) + \ddn \times \tq_n - \fn d \times \fn q 
\right )}
\\
& \quad+ \inttett{ \left ( \mathcal{A}_n (\vvn,\ddn)  , \begin{pmatrix}
0\\ \fn d \times \f a_{n}
\end{pmatrix}\right ) +\left (\tilde{ \mathcal{A}}_n (\fn v,\fn d)  , \begin{pmatrix}
0\\ R_n\f a_{n}-\f a_{n}
\end{pmatrix}\right )} + \inttet{\left (  \t \ddn + \fn d \times  \ddn \times \t \ddn , \f a_{n}  \right ) }\,\numberthis \label{therest}
%\end{split}
\end{align*}
where $\tilde{\mathcal{A}}_n$ is defined in~\eqref{ansch}.
The estimates for the terms in the first two lines on the right-hand side of~\eqref{therest} are similar. Therefore, we exemplify the estimates for the first term on the right-hand side.
We observe after some rearrangements
\begin{align*}
&\inttet{\left (  \fn d \times \f a_{n} , \ddn \times (\vvn \cdot \nabla) \ddn - \fn d \times ( \fn v \cdot \nabla ) \fn d 
\right )}\\
&={}\inttet{\left (  \fn d \times \f a_{n}  , (\ddn -\fn d ) \times (\vvn \cdot \nabla) \ddn +  \fn d  \times (\vvn  \cdot \nabla ) ( \ddn - \fn d)  
\right )}\\
&\quad + {}\inttett{\left (  \fn d \times \f a_{n}   ,  \fn d \times (( \vvn - \fn v) \cdot \nabla) (\fn d- \ddn)
\right )+\left (  \fn d \times \f a_{n}   ,  \fn d \times (( \vvn - \fn v) \cdot \nabla)  \ddn
\right )}
\end{align*}
that this term can be estimated by
\begin{align*}
&\inttet{\left (  \fn d \times \f a_{n}  , \ddn \times (\vvn \cdot \nabla) \ddn - \fn d \times ( \fn v \cdot \nabla ) \fn d 
\right )}\\
&\leq {}   \| \fn d   \|_{L^\infty (\f L^\infty)}\| \nabla \ddn\|_{L^\infty(\f L^3)}  \inttet{ \| \vvn \|_{\f L^\infty } \| \fn d- \ddn \|_{\f L^6} \| \f a_{n}\|_{\f L^2}}+  
\| \fn d \| _{L^\infty(\f L^\infty)}^2 \inttet{
\| \vvn \|_{\f L^\infty }  \| \f a_{n}\|_{\f L^2} \| \nabla \fn d -\nabla \ddn \|_{\Le}}
\\& \quad + \delta \inttet{\| \fn v -\vvn \|_{\f L^6}^2 } %+ C_\delta  \| \fn d   \|_{L^\infty (\f L^\infty)}^4 \| \nabla \ddn\|_{L^\infty(\f L^3)} ^2 \inttet{ \| \ddn \|_{\f L^3} \| \fn d -\ddn \|_{\f L^{12}}^4 }
+  C_\delta \| \fn d   \|_{L^\infty (\f L^\infty)}^4 \inttet{\| \f a_{n}\|_{\f L^3}^2 \| \nabla \fn d - \nabla \ddn \|_{\Le}^2 }+  C_\delta \| \fn d   \|_{L^\infty (\f L^\infty)}^4  \inttet{\| \nabla \ddn \|_{\f L^3}\| \f a_{n}\|_{\Le}^2 } \,.
\end{align*}
For the abbreviation $\f a_{n}$ (see~\eqref{anm}), we observe
\begin{align}
\| \f a_{n}(t) \|_{\f L^2}^2 \leq{}& \frac{1}{k} \| \ddn \| _{ L^\infty(\f L^\infty)}^2  \| \nabla \ddn (t) \|_{\f L^6}^4 \| \fn d(t) - \ddn(t) \|_{\f L^6}^2 + \left ( \chi_{\|}+ \chi_{\bot}\right ) \| \ddn \|_{\f L^\infty} ^2 \|\HH \|_{\f L^\infty} ^2  \| \f H - \HH \|_{\f L^2}^2 \,,\label{anmL2}
\intertext{and}
\begin{split}
\| \f a_{n}(t) \|_{\f L^3}^2 \leq{}& \frac{1}{k} \| \ddn \| _{ L^\infty(\f L^\infty)}^2  \| \nabla \ddn (t) \|_{\f L^6}^4\left (  \| \fn d\|_{L^\infty(\f L^\infty)} + \|  \ddn  \|_{L^\infty(\f L^\infty)}\right )^2  
\\& 
+ \left ( \chi_{\|}+ \chi_{\bot}\right ) \| \ddn \|_{ L^\infty(\f L^\infty)} ^2 \|\HH \|_{\f L^\infty} ^2 \left (\| \f H\|_{\f L^3}+ \|\HH \|_{\f L^3}\right ) ^2 \,. 
\end{split}\label{anmL3}
\end{align}
From where we may infer that
\begin{multline*}
\inttet{\left (  \fn d \times \f a_{n}  , \ddn \times (\vvn \cdot \nabla) \ddn - \fn d \times ( \fn v \cdot \nabla ) \fn d 
\right )}
\leq {}\\   \inttet{\mathcal{K} \left (\mathcal{E}(\fn v ,\fn d,  \f H| \vvn , \ddn , \HH )+ \| \f H - \HH\|_{\Le}^2  \right )} + \delta \inttet{\mathcal{W}(\fn v, \fn d | \vvn, \ddn )}  
\,.
\end{multline*}
%Note that in the case $\f H = \HH$, $\f a_{n} $ does not depend on $\f H$ and $\HH$, respectively.
Again, we exemplified the calculations for the first terms, the other follow similar such that we observe from~\eqref{therest} 
 \begin{align*}
 \inttet{\left (\t \ddn - \t \fn d , 
%\fn d -\ddn
\f a_{n} \right ) }{}
 \leq {}& \inttet{\mathcal{K}\left (\mathcal{E}(\fn v ,\fn d,\f H  | \vvn , \ddn,  \HH)+ \| \f H - \HH\|_{\f L^2 }^2 \right ) }+ \delta \inttet{\mathcal{W}(\fn v ,\fn d | \vvn, \ddn)}   \\&+ \inttet{ \left [\left ( \mathcal{A}(\vvn,\ddn )  , \begin{pmatrix}
 0 \\ \fn d \times \f a_{n}
 \end{pmatrix} \right ) +\left ({ \mathcal{A}}_n (\fn v,\fn d)  , \begin{pmatrix}
0\\\fn d \times (R_n \f a_{n}- \fn a )
\end{pmatrix}\right )\right ]} \\&+ \inttett{ ( \t \fn d (|\fn d|^2-1), R_n \fn a- \fn a) +\left (  \t \ddn + \fn d \times  \ddn \times \t \ddn , \f a_{n}  \right )}
\,.
 \end{align*}
Note that $\tilde{\mathcal{A}}_n (\fn v ,\fn d) \cdot ( 0,R_n\f a_{n}-\fn a )^T = \mathcal{A}_n (\fn v ,\fn d) \cdot ( 0,\fn d \times(R_n \f a_{n}-\fn a) )^T + ( \t \fn d (|\fn d|^2-1), R_n \fn a- \fn a) $ due to the properties of the cross product stated in Section~\ref{sec:not} and the property $\t \fn d \cdot \fn d=0$ of the solution to the approximate scheme (compare to~\eqref{tdumform}).

The last term on the right hand side can be transformed via~\eqref{tdumform} to
\begin{align*}
\inttet{\left (  \t \ddn + \fn d \times  \ddn \times \t \ddn , \f a_{n}  \right ) }= {} &\inttet{\left (  (1- |\ddn|^2) \t \ddn + \frac{1}{2}\t | \ddn|^2 \ddn ,  \f a_{n}  \right ) } 
\\
& {} +\inttet{\left (   \fn d \times  \ddn \times \t \ddn - \ddn \times \ddn \times \t \ddn ,   \f a_{n} \right ) } \,
%\\
\end{align*}
and the second line can be estimated by 
\begin{align*}
\inttet{\left (   \fn d \times  \ddn \times \t \ddn - \ddn \times \ddn \times \t \ddn ,   \f a_{n} \right ) } ={} &  \inttet{\left (   ( \fn d- \ddn )  \times  \ddn \times \t \ddn  ,   \f a_{n} \right ) }\\  \leq{}& \| \ddn \|_{L^\infty(\f L^\infty)} \inttet{\| \t \ddn \| _{ \f L^3}\| \f a_{n}\|_{\f L^2} \| \fn d - \ddn \|_{ \f L^6} } \,.
\end{align*}
Inserting everything back into~\eqref{therest} yields for the term $I_{14}$ with~\eqref{anmL2}  that
\begin{align*}
I_{14}\leq {}& 
\inttett{ \left ( \mathcal{A}(\vvn,\ddn )  , \begin{pmatrix}
 0 \\ \fn d \times \f a_{n}
 \end{pmatrix} \right ) +\left ({ \mathcal{A}}_n (\fn v,\fn d)  , \begin{pmatrix}
0\\\fn d \times  ( R_n\f a_{n}-\fn  a) 
\end{pmatrix}\right )}\\
&+\inttett{\left (  (1- |\ddn|^2) \t \ddn + \frac{1}{2}\t | \ddn|^2 \ddn ,  \f a_{n}  \right )  + ( \t \fn d ( | \fn d|^2 -1) , R_n \f a_n - \fn a) }  \\
&+ \inttet{\mathcal{K} \left (\mathcal{E}(\fn v ,\fn d,  \f H| \vvn , \ddn , \HH )+ \| \f H - \HH\|_{\f L^2 }^2 \right )} + \delta \inttet{\mathcal{W}(\fn v, \fn d | \vvn, \ddn )}   \,.
\end{align*}
%Note that the term in the second line of the above equality vanishes since $(\fn d - \ddn )\times (\fn d - \ddn )=0$.
In the \textbf{eighth and last step}, Gronwall's lemma is applied yielding the approximate relative energy inequality. 
Inserting all the calculations and estimates back into~\eqref{zurueck}, yields
\begin{align*}
 \frac{1}{2}&\mathcal{E}(\fn v ,\fn d,  \f H| \vvn , \ddn , \HH )(t) + \| \f H - \HH \|_{ \f L^2}^2 +(1-\delta) \inttet{\mathcal{W}(\fn v, \fn d | \vvn, \ddn ) }  \leq{}  \mathcal{D}_0(\fn v ,\fn d,  \f H| \vvn , \ddn , \HH )(0)   \\& + \inttett{\mathcal{K}\left (\mathcal{E}(\fn v ,\fn d,  \f H| \vvn , \ddn , \HH ) +\| \f H - \HH \|_{\f L^2}^2 \right ) +\langle (I-R_n)\t \dd , \f q(\dd) - \f q( \fn d)\rangle}
\\&+\inttett{ \left ( \mathcal{A}(\vvn,\ddn )  , \begin{pmatrix}
 \vvn-\fn v  \\ \fn d \times ( \tq_n - \fn q + \f a_{n})
 \end{pmatrix} \right ) +\left ({ \mathcal{A}}_n (\fn v,\fn d)  , \begin{pmatrix}
P_n \vvn -\vvn \\\fn d \times  ( R_n \f a_{n}-\f a_{n})
\end{pmatrix}\right )}\\&
+ \inttet{
\left [
\left (( 1 - |\ddn|^2 ) \t \ddn + \frac{ 1}{2}\t | \ddn|^2 \ddn, \tq_n-\fn q + \f a_{n}  \right ) + ( \t \fn d ( | \fn d|^2 -1) , R_n \f a_n - \fn a) 
\right ]
 }
\end{align*}
Note that we added $ \| \f H- \HH \|_{ \f L^2}^2$ to both sides of the inequality and used that this term is not time dependent. 
Choosing $\delta =1/2$ allows to absorb $\mathcal{W}$ on the left hand side and Gronwall's Lemma provides the assertion. 

\end{proof}
\begin{proof}[Proof of Theorem~\ref{thm:main}]
We argue that~\eqref{discreterelativeEnergy} converges to~\eqref{relenin} as $n\ra \infty$. First, we observe that for $\dd$ fulfilling $| \dd|=1$ the terms $(1- | \dd|^2) $ and $\t | \dd |^2 $ vanish. 
From~\eqref{regtest},~\eqref{qdefq}, and~\eqref{anmL3}, we may infer $ \| R_n \f a_{n} - \f a_n \|_{\f L^3}\ra 0$ as $n\ra \infty$. 
Additionally, we find $\| P_n \vvn- \vvn \|_{ \Hc \cap\,\V}\ra 0 $ as $n\ra \infty$. All terms in~\eqref{vdis} and~\eqref{ddis} are bounded in $(  \Hc \cap\,\V)^*$ and $\f L^{3/2}$, respectively. 
Therefore, the terms 
\begin{align*}
\int_0^t \Big ({ \mathcal{A}}_n (\fn v,\fn d)  , &\begin{pmatrix}
P_n \vvn -\vvn \\\fn d \times (R_n \f a_{n}-\f a_{n})
\end{pmatrix}\Big )
\leq \\{}& \left ( \| \t \fn v \|_{L^2(( \Hc \cap\,\V)^*)} +c \| \fn v  \|_{L^\infty(\f L^2)}\| \fn v \|_{L^2(\He)}  + \| \f g \|_{L^2( \Vd)} \right ) \| P_n \vvn- \vvn \|_{L^2( \Hc \cap\,\V)} \\
&+\left (2c \| \nabla \fn d \|_{L^\infty(\f L^2)} \| \fn d\|_{L^\infty(\f L^\infty)} ^2 \| \fn q\|_{L^2( \f L^2)} + c\| \f T^L_n\|_{ L^2(\f L^2)} \right ) \| P_n \vvn- \vvn \|_{L^2( \Hc \cap\,\V)} \\
& + \| \fn d \|_{L^\infty( \f L^\infty)}^2 \left ( \| \t \fn d\|_{L^2( L^{3/2})} + \| \fn v \|_{L^2(\f L^6)} \| \fn d \| _{ L^\infty( \He)}  \right ) \| R_n \fn a - \fn a \| _ {L^3} 
\\
& + \| \fn d \|_{L^\infty( \f L^\infty)}^2 \left (  (1+ |\lambda|) \| \fn v \|_{ L^2(\He)} \| \fn d \|_ {L^\infty(\f L^\infty)} + \| \fn q\|_{ L^2( \Le)} \right ) \| R_n \fn a - \fn a \| _ {L^2} 
\,\numberthis \label{abshtvn}
\end{align*}
vanish as $n \ra \infty$. 
Similarly, we find 
\begin{align*}
&\left ( \nabla( R_n -I)\t \dd ; \f \Lambda : \nabla \fn d + \fn d \cdot \f \Theta \dreidots \nabla \fn d \o \fn d \right ) \\
 & +  \left ( (R_n- I) \t \dd , \nabla \fn d  : \f \Theta \dreidots \nabla \fn d - \chi_{\|} \f H ( \fn d \cdot \f H ) + \chi_{\bot} \f H\times ( \f H \times \fn d) \right )+ ( \t \fn d ( | \fn d|^2 -1) , R_n \f a_n - \fn a) \\
& \leq{} \| (R_n - I) \t \dd \|_{\Hb} \left ( |\f \Lambda | \| \fn d \|_{ \He} + | \f \Theta | \| \fn d \|_{ \f L^\infty}^2 \| \fn d \|_{ \He} \right ) \\
&\quad + \| (R_n - I) \t \dd \|_{ \f L^\infty} \left (| \f \Theta | \| \fn d \|_{ \He}^2 + ( \chi_{ \|} + \chi_{ \bot} ) \| \f H\|_{\Le}^2\right )  \| \fn d\|_{ \f L^\infty}   + \| \t \fn d \|_{ \f L^{3/2} }( \| \fn d \|_{ \f L^\infty} ^2 + 1) \| (R_n - I) \fn a \|_{ \f L^3}\,.
\end{align*}
The right-hand side vanishes as $n \ra \infty$ since the $\Le$-projection $R_n$ is assumed to be stable in $\Hb$, $\f L^\infty$, and thus also in $\f L^3$. 
Note that the boundary values are constant in time such that $\t\dd\in \Hb$. 
The stability of the projection $R_n$ in $\f L^3$ yields that $ \tq_n \ra \tq $ in $\f L^3$ as $n\ra \infty$. 
Due to $| \dd|=1$,  we find for the Ericksen stress that $ \nabla\dd^T ( | \dd |^2 I - \dd \o \dd ) \tq=  \nabla \dd^T \tq$. 
 Due to an integration-by-parts formula~(see~\cite[Equation~21]{unsere}) this is equivalent to the formulation in~\eqref{eq:strong}. 
 Concerning the Leslie stress, we observe by the definition that in general $\tilde{\f T}^L_n \neq \tilde{\f T}^L$. Especially, it holds for $\tilde{\f T}^L$ given in~\eqref{Leslie} and $\tilde{\f T}^L_n$~(see \eqref{lesliedis} and Theorem~\ref{thm:dis})
 \begin{align*}
 \tilde{\f T}^L_n - \tilde{\f T}^L = ( 1 - | \dd|^2) \left ( ( \dd \o \tq ) _{\skw} + \lambda ( \dd \o \tq)_{\sym} \right ) + | \dd|^2 \left ( (\dd \o ( \tq - \tq_n )) _{ \skw} + \lambda ( \dd \o ( \tq -\tq_n ) )_{\sym}\right ) \,.  
 \end{align*}
 Due to $| \dd|=1$ and $\tq_n \ra \tq$ in $L^2(0,T; \f L^3)$, we find $  \tilde{\f T}^L_n \ra  \tilde{\f T}^L$ in $ L^2(0,T; \f L^3)$. 
Taken together, the approximate relative energy inequality~\eqref{discreterelativeEnergy} converges to the continuous relative energy inequality~\eqref{relEnergy}.

Taking in equation~\eqref{ddis} the test function $R_n (\rot{\fn d } ^T \f \zeta) $ for $\f \zeta \in \C^\infty(\Omega \times (0,T)) $,  and adding and subtracting $\rot{\fn d }^T \f \zeta$ yields
\begin{align*}
&\left (  \fn d \times \left ( \t \fn d + | \fn d| ^2\left (( \fn v \cdot \nabla ) \fn d -\skn v \fn d + \lambda \syn v \fn d + \fn q\right ) \right ), \f \zeta \right )\\ 
&=\left (  \t \fn d + \left ( | \fn d|^2 I - \fn d \o \fn d \right )( (\fn v \cdot \nabla ) \fn d - \skn v \fn d + \lambda \syn v\fn d+ \fn q )   , ( I - R_n  )(\fn d \times  \f \zeta )  \right ) 
 \end{align*}
The right-hand side can be estimated using~\eqref{dtdn}
\begin{align*}
& \left (  \t \fn d + \left ( | \fn d|^2 I - \fn d \o \fn d \right )( (\fn v \cdot \nabla ) \fn d - \skn v \fn d + \lambda \syn v\fn d+ \fn q )    , ( R_n - I ) (\fn d \times\f  \zeta)   \right ) 
\\
& \leq \| ( R_n - I )(\fn d \times \f  \zeta )\|_{ L^2(\f L^{3}) } \left (  \| \t \fn d\|_{L^2(\f L^{3/2})}  + \| \fn d \|_{L^\infty(\f L^{\infty})}^2  \|  \fn v\|_{L^2(\f L^6)}  \| \nabla \fn d\|_{ L^\infty(\Le)} \right ) 
\\&
\quad
+ \| ( R_n - I ) (\fn d \times \f  \zeta ) \|_{ L^2(\f L^2) } \| \fn d \|_{L^\infty(\f L^{\infty})}^2   \left ((1+|\lambda|) \| \fn v\|_{L^2(\He)}\| \fn d\|_{L^\infty(\f L^{\infty})}^2 + c  \| \fn d \times \fn q\|_{L^2(\Le)} \right ) \,
\end{align*} 
 such that the right-hand side converges to zero as $n \ra \infty$. 
 Furthermore, we observe 
 \begin{align*}
&\left (  \fn d \times \left ( \t \fn d + | \fn d| ^2\left (( \fn v \cdot \nabla ) \fn d -\skn v \fn d + \lambda \syn v \fn d + \fn q\right ) \right ), \f \zeta \right )\\ 
&=\left (  \fn d \times \left ( \t \fn d +( \fn v \cdot \nabla ) \fn d -\skn v \fn d + \lambda \syn v \fn d + \fn q \right ), \f \zeta \right )\\
&\quad + \left (  \fn d \times \left ( (| \fn d| ^2- 1)\left (( \fn v \cdot \nabla ) \fn d -\skn v \fn d + \lambda \syn v \fn d + \fn q\right ) \right ), \f \zeta \right )\,,
 \end{align*}
 where the last line can again be estimated by
 \begin{multline*}
 \left (  \fn d \times \left ( (| \fn d| ^2- 1)\left (( \fn v \cdot \nabla ) \fn d -\skn v \fn d + \lambda \syn v \fn d + \fn q\right ) \right ), \f \zeta \right )
\\
\leq {} \| \fn d\|_{ \f L^\infty} \left \| | \fn d |^2 -1 \right \|_{ \f L^3} \left (\| \fn v \|_{ \f L^6} \| \fn d \|_{ \He} +( 1+ | \lambda |) \| \fn v \|_{ \He} \| \fn d \|_{ \f L^6} + \| \fn q \|_{ \f L^{3/2}}\right )\,.
 \end{multline*}
 Due to Proposition~\ref{prop:norm} and Lebesgue's theorem of dominated convergence, we observe that $\| | \fn d|^2 -1\|_{ \f L^3}\ra 0$ as $n\ra \infty$ such that  the right-hand side converges to zero. The bound on the time derivative of $\f v $ (see~\eqref{timev}) follows from Proposition~\ref{prop:time}. 
 Taken together, the semi-discrete approximation scheme convergences to a dissipative solution in the sense of Definition~\ref{def:diss}. 
 
\end{proof}

\section{Optimal control\label{sec:opt}}
In this section, we are going to introduce an optimal control problem, where the Ericksen--Leslie equations equipped with the Oseen--Frank energy acts as a constraint. This constraint is inserted via the dissipative solvability concept in form of an inequality constraint and an additional equality constraint due to the director equation~\eqref{eq:mdir}. 
The general goal of this section is to approximate the optimal control problem in the way that the cost functional remains the same throughout the approximation and the dissipative solutions are approximated in the way introduced in the first part of this article.
\subsection{Optimal control problem\label{sec:optcon}}
We introduce an end time problem, \textit{i.e.}, the goal is to ``be near a desired state at a given time $T$ using a small control $\f H$''. 
Consider the following optimal control problem for $\f v_T\in \Ha$ and $\f d_T \in \He$:
\begin{center}
%\fcolorbox{lightgray}{lightgray}{
\begin{minipage}{0.9\linewidth}
\textbf{Continuous optimal control problem}\\
\begin{align}
\min_{ \| \f H \|_{ \f L^3} \leq c_{\f H}}  J( \f v, \f d , \f H ) : = \min_{ \| \f H \|_{ \f L^3} \leq c_{\f H}} \| \f v( T) - \f v_T \|_{\f L^2}^2 + \| \f d (T) - \f d _T \|_{\He}^2 + \gamma \| \f H\|_{\f L^2   }^2 \,.\label{opt1}
\end{align}
 for $\f v_T\in \Ha$ and $\f d_T \in \He$ given
such that there exists an $\f d \times \f q$, which fulfills   together with  $( \f v , \f d )$ the Definition~\ref{def:diss} for the magnetic field~$\f H$, where the dependence of $\mathcal{K}$ in Definition~\ref{def:diss} on the $\| \f H \|_{\f L^3} $-norm is replaced by $c_{\f H}$.% together with Remark~\ref{rem:addition}. 
\end{minipage}
%}
\end{center}
First we want to argue that there exists an optimal control. The proof relies on the standard procedure of variational calculus.
\begin{proposition}\label{prop:opt}
The continuous optimal control problem~\eqref{opt1} possesses an optimal control $\f H^*$ and an associated optimal state $( \f v^*,\f d^*)$. 
\end{proposition}
\begin{remark}
The optimal control $\f H^*$ is not necessarily unique.  Even the associated state $( \f v^*,\f d^*) $ to an optimal control $\f H^*$ may not be unique since the uniqueness for dissipative solutions is not known. 
The uniqueness of the optimal control is not known, even though the cost functional~$J$ in convex in $\f H$, because the solution set fulfilling Definition~\ref{def:diss} is not known to be convex with respect to $\f H$. 
\end{remark}

\begin{proof}
First, we observe that the set of possible solutions is not empty. Indeed  for every $\f H\in\f L^2$ with $\| \f H\|_{ \f L^3} \leq c_{\f H}$, there exists a dissipative solution (see Theorem~\eqref{thm:main}).

In the following, we consider a minimizing sequence $ ( \fn v,\fn d ,\fn H)$ of the optimization problem~\ref{opt1}. 
Choosing  $( \vv ,\dd , \HH) = ( 0, \Sr \f d_1 , 0)$ in the dissipative formulation~\eqref{def:diss} grants \textit{a priori} estimates. 
Here $\f d_1\in \Hrand{s-1}$ with $s\in [5/2,3]$ is the boundary condition (see~\eqref{anfang}) and $\Sr$ is the extension operator (see~Proposition~\ref{prop:fort}).
  Note that $\Sr \f d_1 \in \f H^s$ with $s \geq5/2$ such that $ \Sr \f d_1 \in \f W^{2,3} $.   Extracting weakly converging subsequences provides all convergences in~\eqref{wkonv} and
%\begin{subequations}\label{wkonv2}
\begin{align}
   \fn H & \rightharpoonup \f H &\quad& \text{ in } L^{2} (0,T;\f L^2 )\,.\label{w:H}
\end{align}
%\end{subequations}
The arguments to validate this convergences are essentially the same as in Proposition~\ref{prop:wkonv}. 
Note the additional convergence~\eqref{w:H} due to the boundedness of the cost functional~$J$. 

In the same way as in Corollary~\ref{cor:Cw}, we observe the convergences in~\eqref{conv:weak}. 
This allows to go to the limit in Definition~\ref{def:diss}. Note that the test functions $ ( \vv , \dd )$ remain fixed. Indeed, the relative energy $\mathcal{E}$ (see~\eqref{relEnergy}) is lower-semi-continuous with respect  to the convergences~\eqref{conv:weak}. 
Especially the convergence~\eqref{sr:d} and~\eqref{w:H} imply that
\begin{multline*}
- \chi_{\|} \|\f d\cdot \f H - \dd \cdot \HH \|_{L^2} ^2 - \chi_{\bot} \|\f d \times \f H- \dd \times \HH\|_{\Le}^2 + \| \f H - \HH \|_{ \f L^2}^2 \leq\\ \liminf_{n\ra \infty} \left (-\chi_{\|} \|\fn d\cdot \fn H - \dd \cdot \HH \|_{L^2} ^2 - \chi_{\bot} \| \fn d \times \fn H- \dd \times \HH\|_{\Le}^2+ \| \fn H - \HH \|_{ \f L^2}^2\right )\,.
\end{multline*}
The same holds true for the relative dissipation~$\mathcal{W}$ (see~\eqref{relW}) and the convergences~\eqref{wr:v}-\eqref{wr:dDd}. In the term $\mathcal{A}$, the weak convergences~\eqref{wkonv} and~\eqref{w:H} suffice to go to the limit in $(\fn v, \fn d,\fn d \times  \fn q)$. 
The term $\mathcal{K}$ (see~\eqref{K}) almost exclusively depends on the test functions $( \vv ,\dd )$. To go to the limit in $\mathcal{K}$ it is crucial that we replaced the $\| \f H\|_{ \f  L^3}$-norm by $c_{\f H}$ in~\eqref{opt1}.
%Additionally, the $\f L^3$-norm of $\f H$ appears. Due to the weak convergence in $\He$ (see~\eqref{w:H}), we can extract another subsequence converging strongly in $\f L^3$ such that $ \| \fn H\|_{\f L^3} \ra \| \f H \|_{\f L^3} $. 
Together, we may infer that the limit $( \f v, \f d , \f H)$ fulfills Definition~\ref{def:diss}. From~\eqref{cw:v},~\eqref{cw:d}, and~\eqref{w:H}, we conclude
\begin{align}
J (\f v , \f d , \f H)  \leq \liminf_{n\ra \infty} J(\fn v ,\fn d , \fn H) \,.\label{liminf}
\end{align}
Therefore, the infimum is actually attained and $\f H$ is an optimal control, whereas $(\f v , \f d )$ is the associated optimal state. 
\end{proof}

\subsection{Approximation of the optimal control problem\label{sec:optdis}}

The continuous problem~\eqref{opt1} is approximated by approximating the state equation and state inequality (due to the dissipative solvability concept) in the way executed in Section~\ref{sec:semiconv}. 
The approximate problems read as

\begin{center}
%\fcolorbox{lightgray}{lightgray}{
\begin{minipage}{0.9\linewidth}
\textbf{Approximate optimal control problem}\\
\begin{align}
\min_{\fn H \in Z_n,\, \| \fn H \|_{ \f L^3} \leq c_{ \f H}}  J( \fn v, \fn d , \fn H ) : = \min_{\fn H \in Z_n,\, \| \fn H \|_{ \f L^3} \leq c_{ \f H}}  \| \fn  v( T) - \f v_T \|_{\f L^2}^2 + \| \fn d (T) - \f d _T \|_{\He}^2 + \gamma \| \fn H\|_{\Le }^2 \,.\label{optn}
\end{align}
 for $\f v_T\in \Ha$ and $\f d_T \in \He$ given
such that   $( \fn v , \fn d )$ solves the approximate scheme~\eqref{eq:dis}. 
\end{minipage}
%}
\end{center}
The optimal control problem~\eqref{optn} is an optimal control problem for ordinary differential equations.
Note that the inequality is replaced by an equality and the solution operator becomes differentiable.

%finite dimensional problem, since the admissible set $Z_n$ is of finite dimensions. 
Therefore, we do not comment on the solvability of this problem. 
We rather assume that problem~\eqref{optn} has a solution for every $n\in\N$ and show that a subsequence converges to an optimal control  in a suitable sense.
This suitable sense is rather weak, since we can only show that such sequence converges under additional assumption on the optimal control. Nevertheless, it is possible to show the lower-semi-continuity of the cost functional for every approximating sequence. 
It is remarkable that even though the control enters the system nonlinearly, only weak convergence of the control is needed to go to the limit in the formulation. Thus, boundedness of the control in some $\f L^p$ spaces suffice to use some weak compactness arguments and go to the limit for a subsequence. 
\begin{proposition}\label{prop:attain}
Let for every $n\in\N$ the triple $( \fn v, \fn d , \fn H)$ be a solution to the optimal control problem~\eqref{optn}. 
There exists $( \f v , \f d , \f H)$ fulfilling Definition~\ref{def:diss} and a  subsequence $\{( \fn v, \fn d ,\fn H)\}_{n\in\N}$ converging in the sense of~\eqref{wkonv} and~\eqref{w:H} to $( \f v , \f d , \f H)$. 
Furthermore, the inequality~\eqref{liminf} holds.

Assume additionally that there exists an  optimal control  and state  $( \vv ,\dd, \HH)$ of problem~\ref{opt1} fulfilling additional regularity assumptions, \textit{i.e.}, \eqref{regtest} with $| \dd| =1$   a.e.~in $\Omega\times (0,T)$ and $ \HH \in \f L^\infty $. 
Then the limit $( \f v ,\f d ,\f H)$ of the sequence $\{ ( \fn v, \fn d , \fn H)\} $ is an optimal control and state, \textit{i.e.}, $J( \fn v , \fn d ,\fn H ) \ra J( \f v ,\f d ,\f H ) = J ( \vv, \dd ,\HH)$ as $n\ra \infty$. 
\end{proposition}
\begin{proof}
Since $( \fn v, \fn d )$ solves the approximate scheme~\eqref{eq:dis} for every given $\fn H$ and for all  $n\in\N$, it can be shown in the same fashion as in the proof of Theorem~\ref{thm:dis} that the approximate relative energy inequality~\eqref{discreterelativeEnergy} is fulfilled for $( \fn v ,\fn d)$ with $\f H = \fn H$. 
The boundedness of the cost functional $J$ implies the existence of an $\f H\in \f L^3$ with $ \| \f H \|_{ \f L^3} \leq c_{\f H}$ and the  weak convergence $ \fn H \rightharpoonup \f H$ in $ \f L^2$. With the same line of reasoning as in the proof of Theorem~\ref{thm:main}, we can show that there exists $ ( \f v ,\f d, \f d \times \f q) $, which obey~\eqref{reldiss} and a subsequence converging to this triple in the sense of Theorem~\ref{thm:main} such that $( \f v , \f d ,\f d\times \f q)$ fulfills Definition~\ref{def:diss} for $\f H $. 
The inequality~\eqref{liminf} follows again by lower-semi-continuity with respect to the convergences~\eqref{wkonv} and~\eqref{w:H}.

Assume now that there exists an  optimal control  and state  $( \vv ,\dd, \HH)$ of problem~\eqref{opt1} fulfilling additional regularity assumptions, \textit{i.e.}, \eqref{regtest} with $| \dd| =1$   a.e.~in $\Omega\times (0,T)$ and $ \HH \in \f L^\infty $. 
We define $ \HH_n := R_n \HH$. Then there exists a solution $( \vv_n,\dd_n)$ to the approximate problem~\eqref{eq:dis} with magnetic field~$\HH_n$ for every $n\in\N$. 
The same line of reasoning as in the first part of the proof implies that there exists a subsequence $\{ ( \vv_n , \dd_n)\}_{n\in\N}$ that converges to a dissipative solution for the magnetic field $\HH$. Due to the additional regularity assumptions on $( \vv,\dd,\HH)$, this dissipative solution is known to be unique such that from the approximate relative energy inequality~\eqref{discreterelativeEnergy} and $\HH_n \ra \HH $ in $\f L^2$, we may infer that $ ( \vv_n ,\dd_n,\HH_n) \ra ( \vv ,\dd ,\HH)$ in $ L^2(0,T; \V )\times L^2(0,T; \He)\times \Le $ as $n\ra \infty$. In this case even a stronger convergence result follows, inserting $( \vv, \dd ,\HH)$ as test functions in the approximate relative energy inequality~\eqref{discreterelativeEnergy} even grants norm convergence such that  (see the proof of Theorem~\ref{thm:main})
\begin{align}
\vv_n & \ra \vv &\quad& \text{ in } \C ([0,T];\Ha) &\quad &\text{and}& \quad&
 \dd_n \ra \dd &\quad& \text{in }\C([0,T];\He)\,.\label{convinC}
\end{align}

Observe that $( \vv_n,\dd_n,\HH_n)$ is a solution to the approximate optimal control problem~\eqref{optn} such that it holds
\begin{align*}
J( \fn v ,\fn d ,\fn H) \leq J( \vv_n,\dd_n,\HH_n)
\end{align*}
for all $n\in\N$, where $( \fn v ,\fn d , \fn H)$ is assumed to solve~\eqref{optn}.
From the stronger convergence~\eqref{convinC} and $R_n \HH \ra \HH$ in $\f L^2$ we may conclude for the cost functional that $ J( \vv_n,\dd_n,\HH_n) \ra J( \vv,\dd,\HH)$. 
The inequality~\eqref{liminf} and the observation that $J( \vv,\dd,\HH)$ is actually a minimum of~\eqref{opt1} imply the assertion. %that the limit  $( \f v ,\f d ,\f H)$ of the sequence $\{ ( \fn v, \fn d , \fn H)\} $ is actually an optimal control with associated optimal state. 
\end{proof}

\begin{remark}
Without this additional regularity assumption in Proposition~\ref{prop:attain}, the solution of the problem is not known to be unique. In this case,
it is not clear what convergence to the solution means.
% another concept of convergence of solutions to this problem has to be considered.
  It may be possible to consider the convergence of the sets of solutions to each other, for the convergence of the control parameter. But a difficulty arises since the solutions sets are not expected to be convex. 

Thus, the weak-strong optimal control scheme concept seems to be the best possibility up to now.
\end{remark}

\addcontentsline{toc}{section}{References}

\small
\bibliographystyle{abbrv}
%\bibliography{Ericksen-Leslie-Version1}

\end{document}